\documentclass[12pt]{article}
\usepackage[normal]{subfigure}
\usepackage{graphicx}
\usepackage{amssymb}
\usepackage{amscd}
\usepackage{amsmath}
\usepackage{amsfonts}
\usepackage{amsbsy}
\usepackage{epsfig}
\usepackage{cite}
\usepackage{amsthm}
\usepackage{paralist}
\numberwithin{equation}{section} \numberwithin{figure}{section}
\newtheoremstyle{theorem}{6pt}{6pt}{\itshape}{}{\bfseries}{.}{.5em}{}
\newtheoremstyle{definition}{6pt}{6pt}{\upshape}{}{\bfseries}{.}{.5em}{}
\theoremstyle{theorem}
\newtheorem{theorem}{Theorem}[section]
\newtheorem{lemma}{Lemma}[section]

\theoremstyle{definition}

\allowdisplaybreaks[4]
\begin{document}
\date{}

\title{Limit cycle bifurcations near a double homoclinic loop with a nilpotent saddle of order 2
\thanks{The project was supported by National Natural Science Foundation of China (11271261 and 11431008).}}

\author{Huanhuan Tian\thanks{Author for correspondence: tianhuanhuan888@163.com.}\\
\small\emph{ Department of Mathematics, Shanghai Normal University,}\\
\small\emph{ Shanghai 200234,  China }}

 \date{}
 \maketitle

 \noindent{\bf Abstract:} In this paper, we deal with limit cycle bifurcations near a double homoclinic loop with a nilpotent saddle of order 2 by studying expansions of the first order Melnikov functions near the loop and coefficients in these expansions. More precisely, we prove that the perturbed system can have 11, 13, 14 or 16 limit cycles in a neighborhood of the loop under certain conditions. Finally, we give an example to illustrate the effectiveness of our main results.

\vspace{0.05in}



\section{Introduction}
Consider a near-Hamiltonian system of the form
\begin{equation}\label{1}
\dot{x}=H_y+\epsilon p(x,y,\delta),\ \
\dot{y}=-H_x+\epsilon q(x,y,\delta)
\end{equation}
where $H(x,y)$, $p(x,y,\delta)$, $q(x,y,\delta)$ are $C^\infty$ functions, $\delta\in D\subset\mathbb{R}^m$ with $D$ a compact set and $\epsilon\geq0$ is a small parameter. For $\epsilon=0$, \eqref{1} becomes a Hamiltonian system
\begin{equation}\label{2}
\dot{x}=H_y,\ \ \dot{y}=-H_x.
\end{equation}
For \eqref{2} we assume that there exists a family of periodic orbits $L_{h}, h\in J$ defined by $H(x,y)=h$ with J an open interval and $\{L_h\}$ has a center or an invariant curve as its boundary. For \eqref{1}, the main task is to study the number of limit cycles  bifurcated from the periodic orbits $\{L_{h}\}$ of the unperturbed system \eqref{2}. For this case, the first order Melnikov function
\begin{equation}\nonumber
M(h,\delta)=\oint_{L_{h}}qdx-pdy
\end{equation}
plays an important role. See \cite{c3,c4}.

When the center is elementary and the invariant curve is a homoclinic loop or a double homoclinic loop with a hyperbolic saddle, or a heteroclinic loop with more than one hyperbolic saddle, there have been many contributions in studying the corresponding bifurcation of limit cycles. See \cite{c1,c5,c6,c7,c8,c9,c10,c11,c12,c13,c21} and references therein. When the critic  point is nilpotent and  at the origin, by  \cite{c1}, without loss of generality,  $H(x,y)$ can be expanded as
\begin{equation}\label{12}
H(x,y)=\frac{1}{2}y^2+\sum_{i+j\geq3}h_{i,j}x^iy^j
\end{equation}
for $(x,y)$ near $(0,0)$. And by the implicit function theorem, there exists a unique $C^\infty$ function $\varphi(x)$ such that $H_y(x,\varphi(x))=0$ for $|x|$ sufficiently small. Thus, we have
\begin{equation}\label{13}
H(x,\varphi(x))=\sum\limits_{j\geq k}h_jx^j,\ \ h_k\neq 0,\ \ k\geq3
\end{equation}
in which $k$, $h_k$ can be used to judge the type of the singularity and the expressions of $h_j,\ j=3,4,\dots,14$ are shown in the Appendix. For more references, see Section 3.4 in \cite{c1} or \cite{c14}.
For the nilpotent critical point bifurcations, there also have been many works. See \cite{c1,c2,c15,c16,c17,c18,c19,c20,c22,c23} for instance. In \eqref{13}, if $h_3\neq0$, the origin is a cusp of order 1. Limit cycle bifurcations of system \eqref{1} near a homoclinic loop passing through it has been studied in \cite{c15}. And 3, 5 or 6 limit cycles are gotten under some conditions.

If $h_3=0,\ h_4>0$, then the origin is a nilpotent center of order 1. In this case, the asymptotic expansion of $M$ at $h=0$ has been studied in \cite{c16}, which can be used to find the limit cycles of system \eqref{1} in a neighborhood of the origin.

When $h_3=0,\ h_4<0$, the origin is a nilpotent saddle of order 1.  The works in \cite{c1,c2} studied the number of limit cycles for  system \eqref{1} near the loop which is  homoclinic or double homoclinic. As a result, on some conditions there can exist 1,2,3,4,5,6,8,10 or 12 limit cycles.

When $h_3=h_4=0$, $h_5\neq0$, the origin is a cusp of order 2, which is a more degenerate case. In \cite{c17} and \cite{c18}, it was proved that 3,5,7,9,10 or 12 limit cycles can exist for system \eqref{1} near the homoclinic loop which passes through the origin.

If $h_3=h_4=h_5=0$, $h_6<0$, the origin is a nilpotent saddle of order 2.  In this paper, we will study the number of limit cycles bifurcated from a double homoclinic loop with a nilpotent saddle of order 2. From \cite{c2}, we know that a double homoclinic loop with a nilpotent saddle consists of two homoclinic loops both of cuspidal type or smooth type. In view of the similarity in proof and effectiveness of finding more limit cycles, we just study the problem for  cuspidal type (as shown in Figure 1). Our main results and proof are presented in the following two sections.
\begin{figure}[t]
\centering
{\includegraphics[width=7cm]{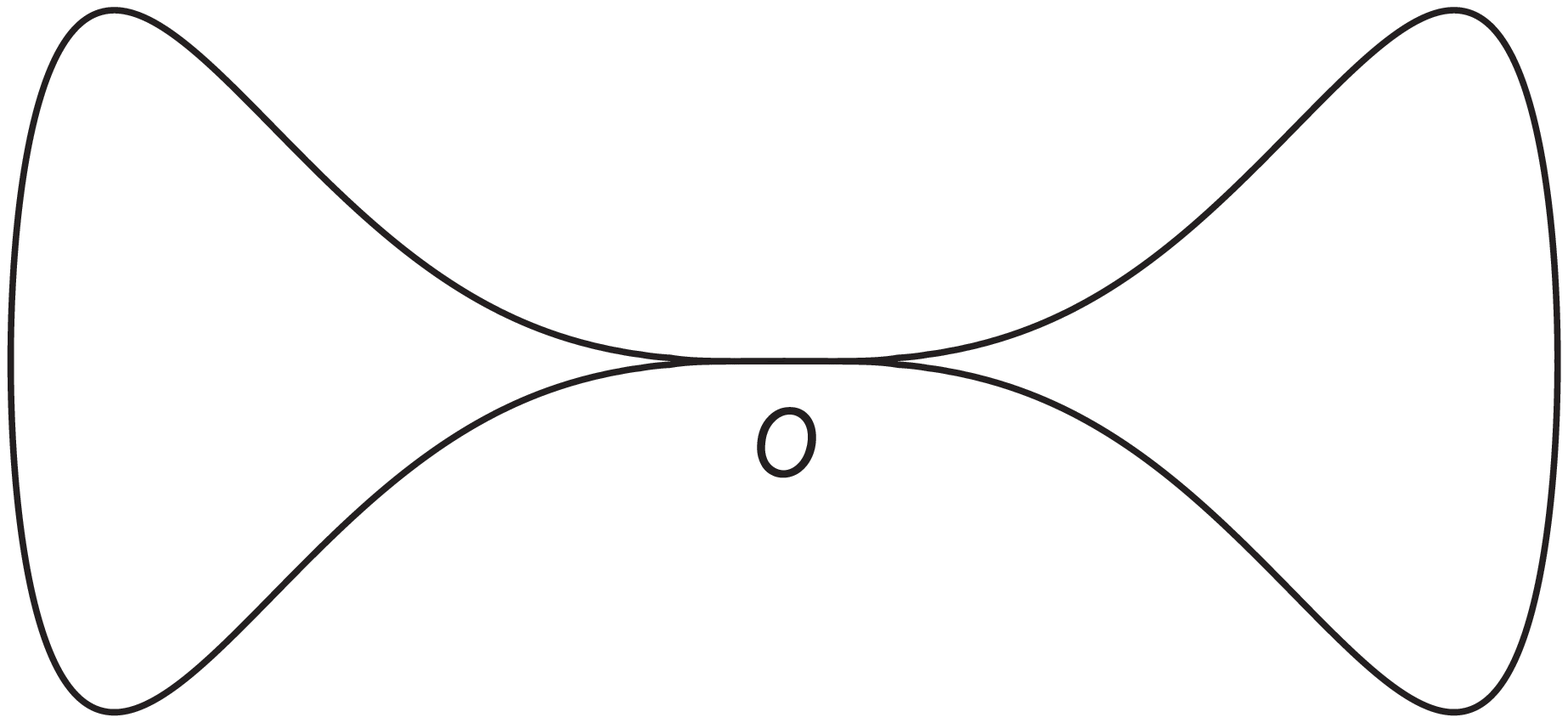}}\\
{Figure 1. \ Double homoclinic loop of cuspidal type with a nilpotent saddle}
\end{figure}
\vspace{0.15in}
\section{Main results and proof}
Now we  suppose that  the unperturbed system  \eqref{2} has a nilpotent saddle of order 2 at the origin and a double homoclinic loop of cuspidal type $L_0^*=L_0\cup\tilde L_0$ defined by $H(x,y)=0$ passes through it, where  $L_0=L_0^*|_{x\geq0},\ \tilde L_0=L_0^*|_{x\leq0}$. Then we can assume that  \eqref{12} and \eqref{13} hold with
\begin{equation}\label{3}
h_3=h_4=h_5=0,\ \ h_6<0.
\end{equation}
Otherwise, we suppose that the level curves of $H(x,y)=h$ define three families of periodic orbits $L_h$, $\tilde L_h$, $L^*_h$ (see Figure 2). By \eqref{12}  we can easily see $L_0^*$ is oriented clockwise and further, by Lemma 3.1.2(i) in \cite{c1}, these periodic orbits can be denoted as $L_h=\{H(x,y)=h,\ 0<-h\ll1,\ x>0\}$, $\tilde L_h=\{H(x,y)=h,\ 0<-h\ll1,\ x<0\}$, $L_h^*=\{H(x,y)=h,\ 0<h\ll1\}$  respectively, which yield three Melnikov functions
\begin{eqnarray}\label{4}
\begin{aligned}
M(h,\delta)=\oint_{L_{h}}qdx-pdy,\\
\tilde M(h,\delta)=\oint_{\tilde L_{h}}qdx-pdy,\\
M^*(h,\delta)=\oint_{L^*_{h}}qdx-pdy.
\end{aligned}
\end{eqnarray}
\begin{figure}[t]
\centering
{\includegraphics[width=7cm]{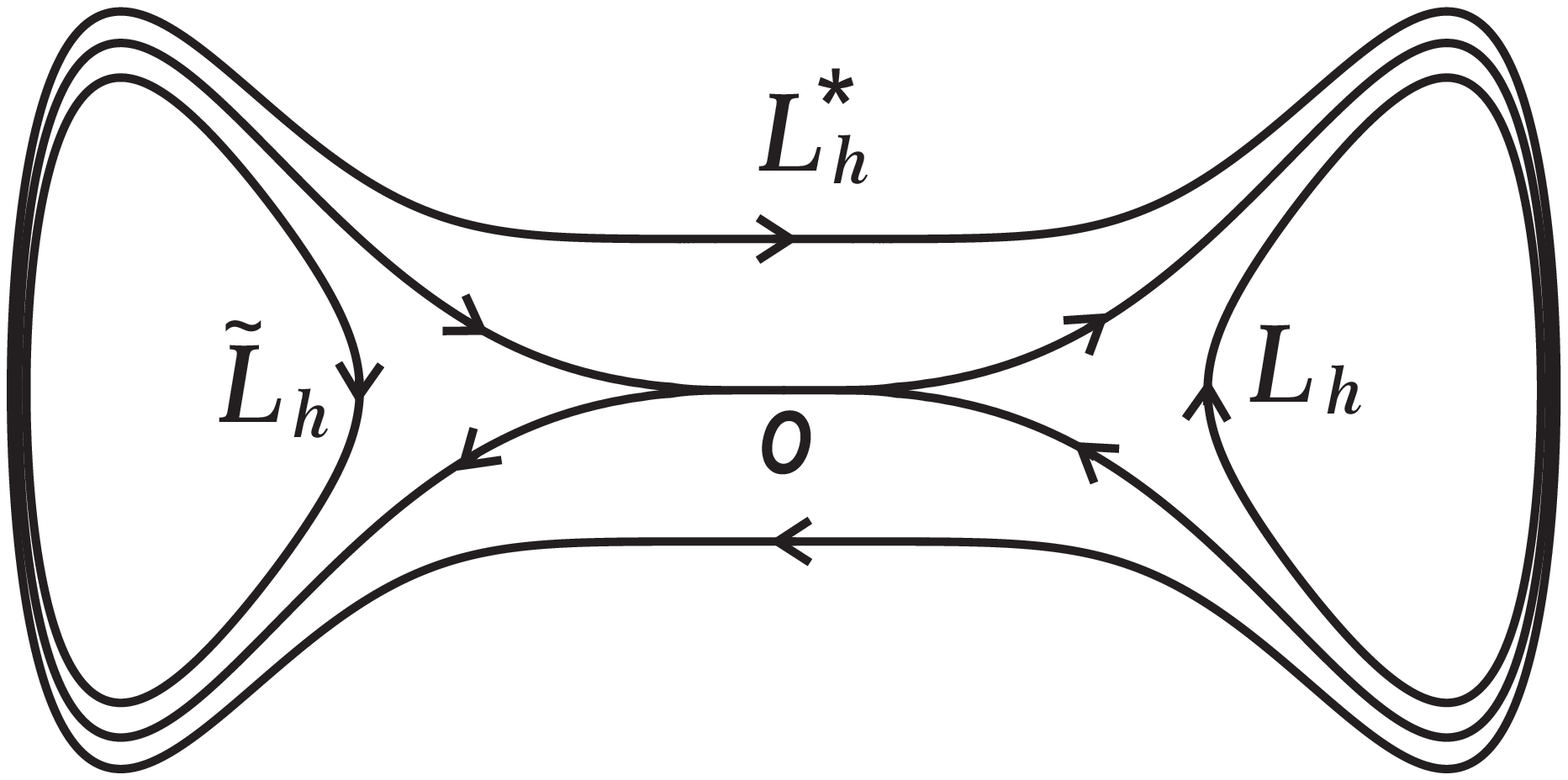}}\\
{Figure 2.\ Three families of periodic orbits near the loop}
\end{figure}
We will give their expressions in the following  lemma by combining with  Theorem 2.1, Lemma 2.2, Lemma 2.3, Lemma 2.5 and Lemma 2.6 in \cite{c2}. Let
\begin{equation}\label{14}
p(x,y,\delta)=\sum\limits_{i+j\geq0}a_{ij}x^iy^j,\ \ q(x,y,\delta)=\sum\limits_{i+j\geq0}b_{ij}x^iy^j.
\end{equation}
\begin{lemma}\label{l1}
Consider systems \eqref{1} and suppose that system \eqref{2} has a double homoclinic loop $L_0^*$ of cuspidal type as stated before. Let \eqref{12}, \eqref{13}, \eqref{3} and \eqref{14} hold. Then for the three Melnikov functions in \eqref{4}, we have
\begin{eqnarray}\nonumber
M(h,\delta)&=&\varphi(h,\delta)-\frac{h\ln|h|}{12}I^*_{12}(h)+{|h|}^{\frac{2}{3}}(\tilde A_0I^*_{10}(h)+\tilde A_1I^*_{11}(h){|h|}^{\frac{1}{6}}+\tilde A_3I^*_{13}(h){|h|}^{\frac{1}{2}}\\ \nonumber
&&+\tilde A_4I^*_{14}(h){|h|}^{\frac{2}{3}}),\ \ \ \ \ \   {\rm for}\  0<-h\ll1,\\ \nonumber
\tilde M(h,\delta)&=&\tilde \varphi(h,\delta)-\frac{h\ln|h|}{12}I^*_{12}(h)+{|h|}^{\frac{2}{3}}(\tilde A_0I^*_{10}(h)-\tilde A_1I^*_{11}(h){|h|}^{\frac{1}{6}}-\tilde A_3I^*_{13}(h){|h|}^{\frac{1}{2}}\\ \nonumber
&&+\tilde A_4I^*_{14}(h){|h|}^{\frac{2}{3}}),\ \ \ \ \ \ {\rm for}\ 0<-h\ll1,\\ \nonumber
M^*(h,\delta)&=&\varphi^*(h,\delta)-\frac{h\ln h}{6}J^*_{11}(h)+2\bar A_0 J^*_{10}(h)h^{\frac{2}{3}}+2\bar A_2 J^*_{12}(h)h^{\frac{4}{3}},\ \ {\rm for}\  0<h\ll1
\end{eqnarray}
where $\varphi(h,\delta),\tilde \varphi(h,\delta) \ \rm{and} \ \varphi^*(h,\delta)$ are $C^\infty$ functions for $h$ near $h=0$,
\begin{eqnarray}\label{15}
\begin{aligned}
&\tilde A_0=-\frac{3}{4}\int_0^1\frac{vdv}{\sqrt{1-v^6}}\approx-0.5258182896<0,\\
&\tilde A_1=-\frac{3}{5}\int_0^1\frac{dv}{\sqrt{1-v^6}}\approx-0.7285951942<0,\\
&\tilde A_3=-\frac{3}{7}\left[\int_0^1\frac{v^4dv}{\sqrt{1-v^6}(1+\sqrt{1-v^6})}-1\right]\approx0.3200718001>0,\\
&\tilde A_4=-\frac{3}{8}\left[\int_0^1\frac{v^3dv}{\sqrt{1-v^6}(1+\sqrt{1-v^6})}-\frac{1}{2}\right]\approx0.0808471737>0,\\
&\bar A_0=\frac{3}{4}\int_0^\infty\frac{dv}{\sqrt{1+v^6}}\approx1.051636580>0,\\
&\bar A_2=-\frac{3}{8}\int_0^\infty\frac{vdv}{\sqrt{1+v^6}(v^3+\sqrt{1+v^6})}\approx-0.1616943474<0,
\end{aligned}
\end{eqnarray}
and
\begin{eqnarray}\label{16}
\begin{aligned}
&I^*_{10}=\tilde r_{00}+(\frac{9}{10}\tilde r_{01}-\frac{1}{10}\tilde r_{60})h+O(h^2),\ \ I^*_{11}=\tilde r_{10}+(\frac{9}{11}\tilde r_{11}-\frac{2}{11}\tilde r_{70})h+O(h^2),\\
&I^*_{12}=\tilde r_{20}+(\frac{3}{4}\tilde r_{21}-\frac{1}{4}\tilde r_{80})h+O(h^2),\ \ I^*_{13}=\tilde r_{30}+O(h),\ \ I^*_{14}=\tilde r_{40}+O(h),\\
&J^*_{10}=r^{(1)}_{01}+(\frac{9}{10}r^{(1)}_{03}-\frac{1}{10}r^{(1)}_{31})h+O(h^2),\ \ J^*_{11}=r^{(1)}_{11}+(\frac{3}{4}r^{(1)}_{13}-\frac{1}{4}r^{(1)}_{41})h+O(h^2),\\
&J^*_{12}=r^{(1)}_{21}+O(h)
\end{aligned}
\end{eqnarray}
where  $\tilde r _{kl}$, $r^{(1)}_{kl}$, $k,l\geq0$  are functions of $a_{ij},\ b_{ij}$, $h_{ij}$. See the Appendix for their concrete expressions.
\end{lemma}
By Lemma \ref{l1}, we will study the expansions of the three Melnikov functions at $h=0$ and formulas of coefficients in these expansions in the following theorem.
\begin{theorem}\label{t1}
Consider system \eqref{1}. Let \eqref{12}, \eqref{13}, \eqref{3}, \eqref{14} hold, and $L_0^*=L_0\cup\tilde L_0$ be a double homoclinic loop of cuspidal type defined by $H(x,y)=0$  where $L_0=L_0^*|_{x\geq0},\ \tilde L_0=L_0^*|_{x\leq0}$. Then for the three functions in \eqref{4}, we have
\begin{eqnarray}\label{5.1}
M(h,\delta)\!\!\!&=\!\!\!&c_0+c_1|h|^{\frac{2}{3}}+c_2|h|^{\frac{5}{6}}+c_3h\ln|h|+c_4h+c_5|h|^{\frac{7}{6}}+c_6|h|^{\frac{4}{3}}+c_7|h|^{\frac{5}{3}}\nonumber\\
\!\!\!&\!\!\!&+c_8|h|^{\frac{11}{6}}+c_9h^2\ln|h|+O(h^2),\ \ \ \ \ \   {\rm for}\  0<-h\ll1,\\ \label{5.2}
\tilde M(h,\delta)\!\!\!&=\!\!\!&\tilde c_0+c_1|h|^{\frac{2}{3}}-c_2|h|^{\frac{5}{6}}+c_3h\ln|h|+\tilde c_4h-c_5|h|^{\frac{7}{6}}+c_6|h|^{\frac{4}{3}}+c_7|h|^{\frac{5}{3}}\nonumber\\
\!\!\!&\!\!\!&-c_8|h|^{\frac{11}{6}}+c_9h^2\ln|h|+O(h^2),\ \ \ \ \ \   {\rm for}\  0<-h\ll1,\\ \label{5.3}
M^*(h,\delta)\!\!\!&=\!\!\!&c^*_0+2c_1^*h^{\frac{2}{3}}+2c_2^*h\ln h+c_3^*h+2c_4^*h^{\frac{4}{3}}+2c_5^*h^{\frac{5}{3}}+2c_6^*h^2\ln h\nonumber\\
\!\!\!&\!\!\!&+O(h^2),\ \ \ \ \ \   {\rm for}\  0<h\ll1
\end{eqnarray}
where
\begin{eqnarray}\label{10}
\begin{aligned}
&c_0=M(0,\delta)=\oint_{L_0}qdx-pdy,\ \ \tilde c_0=\tilde M(0,\delta)=\oint_{\tilde L_0}qdx-pdy,\\
&c_1=\tilde A_0\tilde r_{00},\ \ c_2=\tilde A_1\tilde r_{10},\ \ c_3=-\frac{1}{12}\tilde r_{20},\\
&c_4=\oint_{L_0}(p_x+q_y-\sigma_0-\sigma_1x-\sigma_3x^2)dt+O(|c_1|+|c_2|+|c_3|),\ \ {\rm if}\  \sigma_2=0,\\
&\tilde c_4=\oint_{\tilde L_0}(p_x+q_y-\sigma_0-\sigma_1x-\sigma_3x^2)dt+O(|c_1|+|c_2|+|c_3|),\ \ {\rm if} \ \sigma_2=0,\\
&c_5=\tilde A_3\tilde r_{30},\ \ c_6=\tilde A_4\tilde r_{40},\ \ c_7=-\tilde A_0(\frac{9}{10}\tilde r_{01}-\frac{1}{10}\tilde r_{60}), \\
&c_8=-\tilde A_1(\frac{9}{11}\tilde r_{11}-\frac{2}{11}\tilde r_{70}),\ \ c_9=-\frac{1}{12}(\frac{3}{4}\tilde r_{21}-\frac{1}{4}\tilde r_{80}),\\
&c_0^*=c_0+\tilde c_0,\ \ c_1^*=\bar A_0r_{01}^{(1)},\ \ c_2^*=-\frac{1}{12}r_{11}^{(1)},\\
&c_3^*=\oint_{L^*_0}(p_x+q_y-\sigma_0-\sigma_1x-\sigma_3x^2)dt+O(|c_1|+|c_2|+|c_3|),\ \  {\rm if} \ \sigma_2=0,\\
&c_4^*=\bar A_2r_{21}^{(1)},\ \ c_5^*=\bar A_0(\frac{9}{10}r_{03}^{(1)}-\frac{1}{10}r_{31}^{(1)}),\ \ c_6^*=-\frac{1}{12}(\frac{3}{4}r_{13}^{(1)}-\frac{1}{4}r_{41}^{(1)}),
\end{aligned}
\end{eqnarray}
$\sigma_0=(p_x+q_y)|_{x=y=0},\ \ \sigma_1=(p_{xx}+q_{yx})|_{x=y=0},\ \ \sigma_2=(p_{xy}+q_{yy})|_{x=y=0},\ \  \sigma_3=\frac{1}{2}(p_{xxx}+q_{yxx})|_{x=y=0}$ and $\tilde r _{ij}$,  $r^{(1)}_{kl}$, $i,j,k,l\geq0$ appear in \eqref{16}, $\tilde A_0, \tilde A_1, \tilde A_3, \tilde A_4, \bar A_0, \bar A_2$ appear in \eqref{15}.
\end{theorem}
\begin{proof}
By Lemma \ref{l1}, it is easy to see that
\begin{eqnarray*}
c_0&=&\varphi(0,\delta)=M(0,\delta)=\oint_{L_0}qdx-pdy,\\
\tilde c_0&=&\tilde \varphi(0,\delta)=\tilde M(0,\delta)=\oint_{\tilde L_0}qdx-pdy,\\
c_0^*&=&\varphi^*(0,\delta)=M^*(0,\delta)=\oint_{L^*_0}qdx-pdy=c_0+\tilde c_0.
\end{eqnarray*}
Now we study the formulas of $c_4$, $\tilde c_4$, $c_3^*$. By \eqref{14}, it follows that
\begin{eqnarray*}
\sigma_0&=&(p_x+q_y)|_{x=y=0}=a_{10}+b_{01},\\
\sigma_1&=&(p_{xx}+q_{yx})|_{x=y=0}=2a_{20}+b_{11},\\
\sigma_2&=&(p_{xy}+q_{yy})|_{x=y=0}=a_{11}+2b_{02},\\
\sigma_3&=&\frac{1}{2}(p_{xxx}+q_{yxx})|_{x=y=0}=3a_{30}+b_{21}.
\end{eqnarray*}
Then by the Appendix, we can calculate the following formulas
\begin{eqnarray}\nonumber
\tilde r_{00}&=&r_{01}^{(1)}=2\sqrt{2}(-h_6)^{-\frac{1}{6}}\sigma_0,\\ \label{19}
\tilde r_{10}&=&2\sqrt{2}\left[\left(\frac{1}{3}h_7(-h_6)^{-\frac{4}{3}}-h_{1,2}(-h_6)^{-\frac{1}{3}}\right)\sigma_0+(-h_6)^{-\frac{1}{3}}\sigma_1\right],\\ \nonumber
\tilde r_{20}&=&r_{11}^{(1)}=\frac{\sqrt{2}}{4}(-h_6)^{-\frac{5}{2}}[(24h_6^2h_{0,3}h_{2,1}+12h_6^2h_{1,2}^2-8h_{6}^2h_{2,2}+4h_6h_7h_{1,2}-4h_6h_8\\ \nonumber
&&+3h_7^2)\sigma_0+(-8h_6^2h_{1,2}-4h_6h_7)\sigma_1-8h_6^2h_{2,1}\sigma_2+8h_6^2\sigma_3]
\end{eqnarray}
which can deduce the expressions of $c_1,\ c_2,\ c_3,\ c_1^*,\ c_2^*$.
Let $p_x+q_y=\sigma_0+\sigma_1x+\sigma_2y+\sigma_3x^2+x^3f(x)+xyg(x,y)+y^2h(y)$. By Lemma 3.1.2(ii) in \cite{c1}, we have
\begin{eqnarray}\label{6}
M(h,\delta)&=&M(0,\delta)+\int_0^hM_h(h,\delta)dh=M(0,\delta)+\int_0^h\left[\oint_{L_h}(p_x+q_y)dt\right]dh\nonumber\\
&=&M(0,\delta)+\sigma_0m_0(h)+\sigma_1m_1(h)+\sigma_2m_2(h)+\sigma_3m_3(h)\\
&&+m_4(h)+m_5(h)+m_6(h)\nonumber
\end{eqnarray}
where
\begin{eqnarray*}
m_0(h)&=&\int_0^h\left(\oint_{L_h}dt\right)dh,\ \ m_1(h)=\int_0^h\left(\oint_{L_h}xdt\right)dh,\\
m_2(h)&=&\int_0^h\left(\oint_{L_h}ydt\right)dh,\ \ m_3(h)=\int_0^h\left(\oint_{L_h}x^2dt\right)dh,\\
m_4(h)&=&\int_0^h\left(\oint_{L_h}x^3f(x)dt\right)dh,\ \ m_5(h)=\int_0^h\left(\oint_{L_h}xyg(x,y)dt\right)dh,\\
m_6(h)&=&\int_0^h\left(\oint_{L_h}y^2h(y)dt\right)dh.
\end{eqnarray*}
It is obvious that
\begin{eqnarray}\label{7}
(m_4+m_5+m_6)'&=&\oint_{L_h}(x^3f(x)+xyg(x,y)+y^2h(y))dt\nonumber\\
&=&\oint_{L_h}(p_x+q_y-\sigma_0-\sigma_1x-\sigma_2y-\sigma_3x^2)dt.
\end{eqnarray}
Next, taking $p=0$ and $q=y,\,xy,\,\frac{1}{2}y^2,\,x^2y,\,x^3f(x)y,\,\int_0^yxvg(x,v)dv,\,\int_0^yv^2h(v)dv$ in \eqref{6} respectively, we obtain
\begin{eqnarray}\nonumber
m_0(h)&=&\oint_{L_h}ydx-\oint_{L_0}ydx,\\ \nonumber
m_1(h)&=&\oint_{L_h}xydx-\oint_{L_0}xydx,\\ \nonumber
m_2(h)&=&\oint_{L_h}\frac{1}{2}y^2dx-\oint_{L_0}\frac{1}{2}y^2dx,\\ \label{20}
m_3(h)&=&\oint_{L_h}x^2ydx-\oint_{L_0}x^2ydx,\\ \nonumber
m_4(h)&=&\oint_{L_h}x^3f(x)ydx-\oint_{L_0}x^3f(x)ydx,\\ \nonumber
m_5(h)&=&\oint_{L_h}\left[\int_0^yxvg(x,v)dv\right]dx-\oint_{L_0}\left[\int_0^yxvg(x,v)dv\right]dx,\\ \nonumber
m_6(h)&=&\oint_{L_h}\left[\int_0^yv^2h(v)dv\right]dx-\oint_{L_0}\left[\int_0^yv^2h(v)dv\right]dx.
\end{eqnarray}
Then applying formula \eqref{5.1} to the functions $m_j(h)\ (0\leq j\leq6)$ in \eqref{20}, we yields
\begin{eqnarray}\label{8}
m_j(h)=c_{j1}|h|^{\frac{2}{3}}+c_{j2}|h|^{\frac{5}{6}}+c_{j3}h\ln|h|+c_{j4}h+O(|h|^{\frac{7}{6}})
\end{eqnarray}
where $c_{jk}$ are constants. Now, substituting \eqref{8} into \eqref{6} and comparing with \eqref{5.1}, we obtain
$$c_4=\sigma_0c_{04}+\sigma_1c_{14}+\sigma_2c_{24}+\sigma_3c_{34}+c_{44}+c_{54}+c_{64}.$$
Further, by the formulas of $c_1$, $c_2$, $c_3$ in \eqref{10} and \eqref{19}, it is easy to verify
\begin{eqnarray}\label{9}
c_{41}=c_{42}=c_{43}=c_{51}=c_{52}=c_{53}=c_{61}=c_{62}=c_{63}=0.
\end{eqnarray}
Thus, by \eqref{7}, \eqref{8} and \eqref{9}, it follows that
\begin{eqnarray*}
\oint_{L_0}(p_x+q_y-\sigma_0-\sigma_1x-\sigma_2y-\sigma_3x^2)dt&=&\lim\limits_{h\rightarrow 0}(m_4+m_5+m_6)'\\
&=&\lim\limits_{h\rightarrow 0}(c_{44}+c_{54}+c_{64}+O(|h|^{\frac{1}{6}}))\\
&=&c_{44}+c_{54}+c_{64}\in\mathbb{R}.
\end{eqnarray*}
From the above proof, we can see that  $\oint_{L_0}y dt$ is infinite. Thus, the integral $\oint_{L_0}(p_x+q_y-\sigma_0-\sigma_1x-\sigma_3x^2)dt$  is finite if and only if $\sigma_2=0$. Then, for $\sigma_2=0$,  $$c_4=\oint_{L_0}(p_x+q_y-\sigma_0-\sigma_1x-\sigma_3x^2)dt+b_1c_1+b_2c_2+b_3c_3$$ where $b_i$ are constants.
The formula for $\tilde c_4$ can be obtained in the same way. By the similar argument, we can obtain
$$c_3^*=\oint_{L^*_0}(p_x+q_y-\sigma_0-\sigma_1x-\sigma_2y-\sigma_3x^2)dt+\sigma_0c_{03}^*+\sigma_1c_{13}^*+\sigma_2c_{23}^*+\sigma_3c_{33}^*.$$
And also, $\oint_{L^*_0}(p_x+q_y-\sigma_0-\sigma_1x-\sigma_3x^2)dt$ is finite if and only if $\sigma_2=0$. Thus, we get the formulas of $c_3^*$ in \eqref{10}.
This ends the proof.
\end{proof}
By comparing with the expressions of $\tilde r_{ij}$ and $r^{(1)}_{kl}$, $i,j,k,l\geq0$ in the Appendix, we discover that
\begin{eqnarray*}
\begin{aligned}
\tilde r_{00}=r_{01}^{(1)},\ \ \ \ &\tilde r_{20}=r_{11}^{(1)},\ \ &\tilde r_{40}=r_{21}^{(1)},\ \ \ \ &\tilde r_{60}=r_{31}^{(1)}\\
\tilde r_{80}=r_{41}^{(1)},\ \ \ \ &\tilde r_{01}=r_{03}^{(1)},\ \ &\tilde r_{21}=r_{13}^{(1)}.\ \ \ \ &
\end{aligned}
\end{eqnarray*}
So, we have
\begin{eqnarray*}
c_1^*&=&\frac{\bar A_0}{\tilde A_0}c_1=-D_1c_1,\ \ c_2^*=c_3,\ \ c_4^*=\frac{\bar A_2}{\tilde A_4}c_6=-D_2c_6,\\
c_5^*&=&-\frac{\bar A_0}{\tilde A_0}c_7=D_1c_7,\ \ c_6^*=c_9,
\end{eqnarray*}
where $D_1=|\bar A_0|/|\tilde A_0|,\ D_2=|\bar A_2|/|\tilde A_4|$.
Thus
\begin{eqnarray}\label{11}
M^*(h,\delta)&=&c_0^*-2D_1c_1h^{\frac{2}{3}}+2c_3h\ln h+c_3^*h-2D_2c_6h^{\frac{4}{3}}+2D_1c_7h^{\frac{5}{3}}\nonumber\\
&&+2c_9h^2\ln h+O(h^2).
\end{eqnarray}
Let $$c_{41}=\oint_{L_0}(p_x+q_y-\sigma_0-\sigma_1x-\sigma_3x^2)dt,$$$$\tilde c_{41}=\oint_{\tilde L_0}(p_x+q_y-\sigma_0-\sigma_1x-\sigma_3x^2)dt,$$
$$c_{31}^*=\oint_{ L^*}(p_x+q_y-\sigma_0-\sigma_1x-\sigma_3x^2)dt.$$
From the  proof of Theorem \ref{t1}, we will see that  $c_{41}+\tilde c_{41}=c_{31}^*$ if $\sigma_2=0$.
By using Theorem \ref{l1}, we will study the number of limit cycles near $L_0^*$ as follows.
\begin{theorem}\label{t2}
Suppose that  system \eqref{2} has a double homoclinic loop $L_0^*$ and the assumptions in Theorem \ref{t1} are satisfied. If there exist $\delta_0\in\mathbb{R}^m$, $6\leq l\leq9$ such that
$c_l(\delta_0)\neq0,\ \tilde c_0(\delta_0)=c_{41}(\delta_0)=\tilde c_{41}(\delta_0)=c_j(\delta_0)=0,\ \ j=0,1,2,3,5,\dots,l-1$, and
$${\rm rank}\frac{\partial (\tilde c_0,c_{41},\tilde c_{41},c_0,c_1,c_2,c_3,c_5,\cdots,c_{l-1})}{\partial (\delta_1,\delta_2,\delta_3,\delta_4,\delta_5,\delta_6,\delta_7,\delta_8,\cdots,\delta_m)}(\delta_0)=l+2,$$
then system \eqref{1} can have $2l-2\ (l=8,9)$ or $2l-1\ (l=6,7)$ limit cycles near $L_0^*$ for some $(\epsilon,\delta)$ near $(0,\delta_0)$.
\end{theorem}
\begin{proof}
In the following, we only prove the theorem when $l=9$. Then by similar argument, the theorem follows when $l=6,7,8$. For definiteness, we assume $c_9(\delta_0)<0$. By the assumptions, we can choose $\tilde c_0,c_{41},\tilde c_{41},c_0,c_1,c_2,c_3$, $c_5,c_6,c_7,c_{8}$ as free parameters vary near zero. Thus, when $\sigma_2=0$ the formulas of $M,\ \tilde M,\ M^*$ can be rewritten as
\begin{eqnarray*}
M(h,\delta)\!\!\!&=\!\!\!&c_0+c_1|h|^{\frac{2}{3}}+c_2|h|^{\frac{5}{6}}+c_3h\ln|h|+(c_{41}+O(c_1)+O(c_2)+O(c_3))h\\
\!\!\!&\!\!\!&+c_5|h|^{\frac{7}{6}}+c_6|h|^{\frac{4}{3}}+c_7|h|^{\frac{5}{3}}+c_8|h|^{\frac{11}{6}}+\tilde c_9h^2\ln|h|+O(h^2)\\
\!\!\!&\!\!\!&:=f_1(h,\tilde c_0,c_{41},\tilde c_{41},c_0,c_1,c_2,c_3,c_5,c_6,c_7,c_{8}),\ \ \ \  {\rm for}\  0<-h\ll1,\\
\tilde M(h,\delta)\!\!\!&=\!\!\!&\tilde c_0+c_1|h|^{\frac{2}{3}}-c_2|h|^{\frac{5}{6}}+c_3h\ln|h|+ (\tilde c_{41}+O(c_1)+O(c_2)+O(c_3))h\\
\!\!\!&\!\!\!&-c_5|h|^{\frac{7}{6}}+c_6|h|^{\frac{4}{3}}+c_7|h|^{\frac{5}{3}}-c_8|h|^{\frac{11}{6}}+\tilde c_9h^2\ln|h|+O(h^2)\\
\!\!\!&\!\!\!&:=f_2(h,\tilde c_0,c_{41},\tilde c_{41},c_0,c_1,c_2,c_3,c_5,c_6,c_7,c_{8}),\ \ \ \ {\rm for}\  0<-h\ll1,\\
M^*(h,\delta)\!\!\!&=\!\!\!&c_0+\tilde c_0-2D_1c_1h^{\frac{2}{3}}+2c_3h\ln h+(c_{41}+\tilde c_{41}+O(c_1)+O(c_2)+O(c_3))h\\
\!\!\!&\!\!\!&-2D_2c_6h^{\frac{4}{3}}+2D_1c_7h^{\frac{5}{3}}+2\tilde c_9h^2\ln h+O(h^2)\\
\!\!\!&\!\!\!&:=f_3(h,\tilde c_0,c_{41},\tilde c_{41},c_0,c_1,c_2,c_3,c_5,c_6,c_7,c_{8}),\ \ \ \ {\rm for}\  0<h\ll1
\end{eqnarray*}
where $\tilde c_9=c_9(\delta_0)+O(|\tilde c_0,c_{41},\tilde c_{41},c_0,c_1,c_2,c_3,c_5,c_6,c_7,c_{8}|)<0$. It is easy to verify that
\begin{enumerate}
\item[(1)]if $c_0,\tilde c_0\ll-c_1\ll-c_2\ll c_3\ll-c_{41}\ll\tilde c_{41}\ll-c_5\ll c_6\ll-c_7\ll-c_8\ ({\rm or}\ c_8)\ll1$
or $c_0,\tilde c_0\ll-c_1\ll-c_2\ll c_3\ll-\tilde c_{41}\ll c_{41}\ll c_5\ll c_6\ll-c_7\ll-c_8\ ({\rm or}\ c_8)\ll1$, $f_1$, $f_2$ and $f_3$ have 6, 6 and 4 zeros respectively,
\item[(2)]if $c_0\ll-\tilde c_0\ll-c_1\ll-c_2\ll c_3\ll-c_{41}\ll\tilde c_{41}\ll-c_5\ll c_6\ll-c_7\ll-c_8\ ({\rm or}\ c_8)\ll1$
or $c_0\ll-\tilde c_0\ll-c_1\ll-c_2\ll c_3\ll-\tilde c_{41}\ll c_{41}\ll c_5\ll c_6\ll-c_7\ll-c_8\ ({\rm or}\ c_8)\ll1$, $f_1$, $f_2$ and $f_3$ have 6, 5 and 5 zeros respectively,
\item[(3)]if $\tilde c_0 \ll-c_0\ll-c_1\ll-c_2\ll c_3\ll-c_{41}\ll\tilde c_{41}\ll-c_5\ll c_6\ll-c_7\ll-c_8\ ({\rm or}\ c_8)\ll1$
or $\tilde c_0 \ll-c_0\ll-c_1\ll-c_2\ll c_3\ll-\tilde c_{41}\ll c_{41}\ll c_5\ll c_6\ll-c_7\ll-c_8\ ({\rm or}\ c_8)\ll1$, $f_1$, $f_2$ and $f_3$ have 5, 6 and 5 zeros respectively.
\end{enumerate}
In each case, there are 16 limit cycles for system \eqref{1} near $L_0^*$. This ends the proof.
\end{proof}
\section{Application}
In this section, we consider a Li\'{e}nard system of the form
\begin{equation}\label{e1}
\dot{x}=y,\ \
\dot{y}=-8x^5\left(x-\frac{\sqrt 3}{2}\right)\left(x+\frac{\sqrt 3}{2}\right)-\epsilon \left(\sum\limits_{j=0}^{12}a_jx^j\right)y.
\end{equation}
System $\eqref{e1}|_{\epsilon=0}$ is Hamiltonian with
\begin{equation*}
H(x,y)=\frac{1}{2}y^2-x^6+x^8
\end{equation*}
and the phase portrait is shown in Figure 3. It is easy to verify that for the unperturbed system,  $O(0,0)$ is a nilpotent saddle of order 2 ($k=6,\ h_6=-1$ in \eqref{13}) and a double homoclinic loop $L^*_0=L_0\cup \tilde L_0$ defined by $H(x,y)=0$ passing through it, where $L_0=L_0^*|_{x\geq0},\ \tilde L_0=L_0^*|_{x\leq0}$. Also, there is a center $C_1({\sqrt3}/{2},0)$ ($C_2(-{\sqrt3}/2,0)$, resp.)  inside $L_0$ ($\tilde L_0$, resp.).
Then we have the following result.
\begin{figure}[t]
\centering
{\includegraphics[width=7cm]{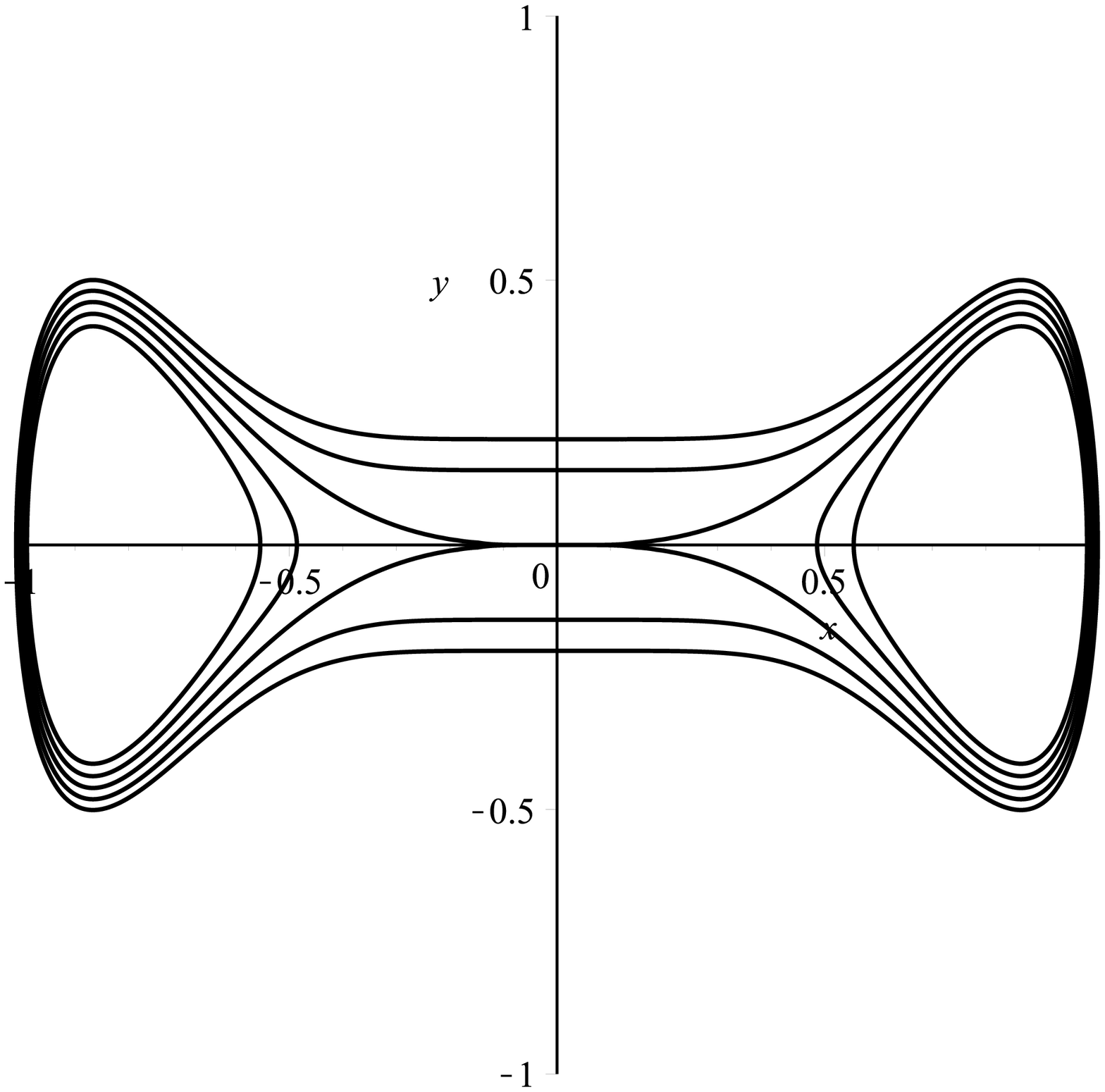}}\\
{Figure 3.\ Phase portrait of system $\eqref{e1}|_{\epsilon=0}$}
\end{figure}
\begin{theorem}
Consider  system \eqref{e1}. Then
\begin{enumerate}
\item[(1)]If $a_{12}\neq0$,  the system \eqref{e1} can have 16 limit cycles near the loop $L_0^*$ for some $(\epsilon,a_0,a_1,\cdots,a_{12})$.
\item[(2)]If $a_{12}=0,\ a_{11}\neq0$,  the system \eqref{e1} can have 14 limit cycles near the loop $L_0^*$ for some $(\epsilon, a_0,a_1,\cdots,a_{11})$.
\item[(3)]If $a_{12}=a_{11}=0,\ a_{10}\neq0$,  the system \eqref{e1} can have 13 limit cycles near the loop $L_0^*$ for some $(\epsilon, a_0,a_1,\cdots,a_{10})$.
\end{enumerate}
\end{theorem}
\begin{proof}
Let $f(x,\delta)=\sum\limits_{j=0}^{12}a_jx^j$. By Theorem \ref{t1}, we obtain
$$c_0(\delta)=M(0,\delta)=-\oint_{L_0}f(x,\delta)ydx=-\sum\limits_{j=0}^{12}a_jI_{1j},$$
$$\tilde c_0(\delta)=\tilde M(0,\delta)=-\oint_{\tilde L_0}f(x,\delta)ydx=-\sum\limits_{j=0}^{12}a_jI_{2j},$$
\begin{eqnarray*}
c_{41}(\delta)\!\!\!&=\!\!\!&\oint_{L_0}(p_x+q_y-\sigma_0-\sigma_1x-\sigma_3x^2)dt=-\oint_{L_0}(f(x,\delta)-a_0-a_1x-a_2x^2)dt\\
\!\!\!&=\!\!\!&-\oint_{L_0}\frac{(f(x,\delta)-a_0-a_1x-a_2x^2)}{y}dx=-2\int_{0}^{1}\frac{(f(x,\delta)-a_0-a_1x-a_2x^2)}{y}dx\\
\!\!\!&=\!\!\!&-2\int_{0}^{1}\frac{1}{y}\left(\sum\limits_{j=3}^{12}a_jx^j\right)dx=-\sqrt{2}\int_{0}^{1}\sum\limits_{j=3}^{12}a_j\frac{x^{j-3}}{\sqrt{1-x^2}}dx\\
\!\!\!&=\!\!\!&-\sqrt{2}\sum\limits_{j=3}^{12}a_j\int_{0}^{1}f_j(x)dx,
\end{eqnarray*}
\begin{eqnarray*}
\tilde c_{41}(\delta)\!\!\!&=\!\!\!&\oint_{\tilde L_0}(p_x+q_y-\sigma_0-\sigma_1x-\sigma_3x^2)dt=-\oint_{\tilde L_0}(f(x,\delta)-a_0-a_1x-a_2x^2)dt\\
\!\!\!&=\!\!\!&-\oint_{\tilde L_0}\frac{(f(x,\delta)-a_0-a_1x-a_2x^2)}{y}dx=-2\int_{-1}^{0}\frac{(f(x,\delta)-a_0-a_1x-a_2x^2)}{y}dx\\
\!\!\!&=\!\!\!&-2\int_{-1}^{0}\frac{1}{y}\left(\sum\limits_{j=3}^{12}a_jx^j\right)dx=-\sqrt{2}\int_{-1}^{0}\sum\limits_{j=3}^{12}a_j\frac{x^{j-3}}{\sqrt{1-x^2}}dx\\
\!\!\!&=\!\!\!&-\sqrt{2}\sum\limits_{j=3}^{12}a_j\int_{-1}^{0}f_j(x)dx
\end{eqnarray*}
where
$$I_{1j}=\oint_{L_0}x^jydx=2\int_{0}^{1}x^jydx=2\sqrt{2}\int_{0}^{1}x^{j+3}\sqrt{1-x^2}dx, \ j=0,1,\ldots,12,$$
$$I_{2j}=\oint_{L_3}x^jydx=2\int_{-1}^{0}x^jydx=2\sqrt{2}\int_{-1}^{0}x^{j+3}\sqrt{1-x^2}dx, \ j=0,1,,\ldots,12,$$
$$f_j(x)=\frac{x^{j-3}}{\sqrt{1-x^2}},\ \ j=3,\ldots,12.$$
Computing them by Maple 17, we obtain
\begin{eqnarray*}
c_0(\delta)\!\!\!&=\!\!\!&-2\sqrt{2}\left[\frac{2048}{109395}a_{12}+{\frac {429}{65536}}\,a_{{11}}\pi +{
\frac {1024}{45045}}\,a_{{10}}+{\frac {33}{4096}}\,a_{{9}}\pi +{\frac
{256}{9009}}\,a_{{8}}+\frac{2}{15}a_0\right.\\
\!\!\!&\!\!\!& \left.+{\frac {21}{2048}}\,a_{{7}}\pi+{\frac {128}{3465}}\,a_{{6}}+{\frac {7}{512}}\,a_{{5}}\pi +{\frac {16}{315}}\,a_{{4
}}+{\frac {5}{256}}\,a_{{3}}\pi+\frac{8}{105}a_2+{\frac {1}{32}}\,a_{{1}}\pi\right],\\
\tilde c_0(\delta)\!\!\!&=\!\!\!&-2\sqrt{2}\left[-\frac{2048}{109395}a_{12}+{\frac {429}{65536}}\,a_{{11}}
\pi -{\frac {1024}{45045}}\,a_{{10}}+{\frac {33}{4096}}\,a_{{9}}\pi -{
\frac {256}{9009}}\,a_{{8}}-\frac{2}{15}a_0\right.\\
\!\!\!&\!\!\!& \left.+{\frac {21}{2048}}\,a_{{7}}\pi-{\frac {128}{3465}}\,a_{{6}}+{\frac {7}{512}}\,a_{{5}}\pi -{\frac {16}{315}}\,
a_{{4}}+{\frac {5}{256}}\,a_{{3}}\pi-\frac{8}{105}a_2+{\frac {1}{32}}\,a_{{1}}\pi\right],\\
c_{41}(\delta)\!\!\!&=\!\!\!&-\sqrt{2}\left[{\frac{128}{315}a_{12}+\frac {35}{256}}\,a_{{11}}\pi +{\frac {16}{35}}
\,a_{{10}}+{\frac {5}{32}}\,a_{{9}}\pi +{\frac {8}{15}}\,a_{{8}}+\frac{3}{16}
\,a_{{7}}\pi +\frac{2}{3}\,a_{{6}}\right.\\
\!\!\!&\!\!\!&\left.+\frac{1}{4}\,a_{{5}}\pi +a_{{4}}+\frac{1}{2}\,a_{{3}}\pi\right],\\
\tilde c_{41}(\delta)\!\!\!&=\!\!\!&-\sqrt{2}\left[{-\frac{128}{315}a_{12}+\frac {35}{256}}\,a_{{11}}\pi -{\frac {16}{35}}
\,a_{{10}}+{\frac {5}{32}}\,a_{{9}}\pi -{\frac {8}{15}}\,a_{{8}}+\frac{3}{16}
\,a_{{7}}\pi -\frac{2}{3}\,a_{{6}}\right.\\
\!\!\!&\!\!\!&\left.+\frac{1}{4}\,a_{{5}}\pi -a_{{4}}+\frac{1}{2}\,a_{{3}}\pi\right].
\end{eqnarray*}
Then by Theorem \ref{t1} and the Appendix, it follows that
\begin{eqnarray*}
c_1(\delta)\!\!\!&=\!\!\!&-2\sqrt{2}\tilde A_0a_0,\ \ c_2(\delta)=-2\sqrt{2}\tilde A_1a_1,\ \ c_3(\delta)=\frac{\sqrt{2}}{12}(a_0+2a_2),\\
c_5(\delta)\!\!\!&=\!\!\!&-\frac{2\sqrt{2}}{3}\tilde A_3(2a_1+3a_3),\ \ c_6(\delta)=-\sqrt{2}\tilde A_4(\frac{55}{36}a_0+\frac{5}{3}a_2+2a_4),\\
 c_7(\delta)\!\!\!&=\!\!\!&-\frac{\sqrt{2}}{10}\tilde A_0(\frac{1729}{648}a_0+\frac{91}{36}a_2+\frac{7}{3}a_4+2a_6),\\
 c_8(\delta)\!\!\!&=\!\!\!&-\frac{4\sqrt{2}}{11}\tilde A_1(\frac{140}{81}a_1+\frac{14}{9}a_3+\frac{4}{3}a_5+a_7),\\
 c_9(\delta)\!\!\!&=\!\!\!&-\frac{\sqrt{2}}{48}(\frac{315}{64}a_0+\frac{35}{8}a_2+\frac{15}{4}a_4+3a_6+2a_8).
\end{eqnarray*}
(1) Suppose that $a_{12}\neq0$. By $\tilde c_0(\delta)= c_{41}(\delta)=\tilde c_{41}(\delta)=c_0(\delta)=c_1(\delta)=c_2(\delta)=c_3(\delta)=c_5(\delta)=c_6(\delta)=c_7(\delta)=c_8(\delta)=0$, we can obtain $$a_i=0,\ i=0,1,2,3,4,6,9,11, \ \ a_5=-\frac{3}{4}a_7,$$
$$ a_8=\frac{21702051851422978291}{27670116110564327424}a_{12},\ \ a_{10}=-\frac{99829438516545753655}{55340232221128654848}a_{12},$$
and  $c_9=-\frac{\sqrt{2}a_8}{24}=-\frac{\sqrt{2}}{24}*\frac{21702051851422978291}{27670116110564327424}a_{12}\neq0.$
Furthermore, $${\rm rank}\frac{\partial (\tilde c_0,c_{41},\tilde c_{41},c_0,c_1,c_2,c_3,c_5,c_6,c_7,c_8)}{\partial (a_0,a_1,a_2,a_3,a_4,a_5,a_6,a_7,a_8,a_9,a_{10},a_{11},a_{12})}=11.$$
Thus, by Theorem \ref{t2}, we obtain 16 limit cycles for system \eqref{e1} near $L_0^*$ for some $(\epsilon, a_0,a_1,a_2,a_3,a_4,a_5,a_6,a_7,a_8,a_9,a_{10},a_{11},a_{12})$ near $$(0,0,0,0,0,0,-\frac{3}{4}a_7,0, a_7,\frac{21702051851422978291}{27670116110564327424}a_{12},0,-\frac{99829438516545753655}{55340232221128654848}a_{12},0,a_{12}).$$
(2) Assume that $a_{12}=0,\ a_{11}\neq0$. The equations $\tilde c_0(\delta)= c_{41}(\delta)=\tilde c_{41}(\delta)=c_0(\delta)=c_1(\delta)=c_2(\delta)=c_3(\delta)=c_5(\delta)=c_6(\delta)=c_7(\delta)=0$ have solution $$a_i=0, \ i=0,1,2,3,4,6,8,10,\ \ a_5=\frac{165}{256}a_{11}-\frac{3}{4}a_7,\ a_9=-\frac{61}{32}a_{11}.$$ And $c_8=-\frac{55\sqrt{2}}{176}\tilde A_1a_{11}\neq0$, $${\rm rank}\frac{\partial (\tilde c_0,c_{41},\tilde c_{41},c_0,c_1,c_2,c_3,c_5,c_6,c_7)}{\partial (a_0,a_1,a_2,a_3,a_4,a_5,a_6,a_7,a_8,a_9,a_{10},a_{11})}=10.$$ This implies that there can have   14 limit cycles for system \eqref{e1} near $L_0^*$ for some $(\epsilon, a_0,a_1,a_2,a_3,a_4,a_5,a_6,a_7,a_8,a_9,a_{10},a_{11})$ near $$(0,0,0,0,0,0,\frac{165}{256}a_{11}-\frac{3}{4}a_7,0,a_7,0,-\frac{61}{32}a_{11},0,a_{11}).$$
(3) Let $a_{12}=a_{11}=0,\ a_{10}\neq0$. By equations $\tilde c_0(\delta)= c_{41}(\delta)=\tilde c_{41}(\delta)=c_0(\delta)=c_1(\delta)=c_2(\delta)=c_3(\delta)=c_5(\delta)=c_6(\delta)=0$, we have
\begin{equation*}\label{18}
a_0=a_1=a_2=a_3=a_4=a_9=0,\ a_5=-\frac{3}{4}a_7,\ a_6=\frac{8}{7}a_{10},\ a_8=-\frac{16}{7}a_{10},
\end{equation*} and $c_7=-\frac{8\sqrt{2}}{35}\tilde A_0a_{10}\neq0.$
Furthermore, we can verify that $${\rm rank}\frac{\partial (\tilde c_0,c_{41},\tilde c_{41},c_0,c_1,c_2,c_3,c_5,c_6)}{\partial (a_0,a_1,a_2,a_3,a_4,a_5,a_6,a_7,a_8,a_9,a_{10})}=9.$$ Thus, by Theorem \ref{t1} we can prove that system \eqref{e1} has 13 limit cycles near $L_0^*$ for some $(\epsilon,a_0,a_1,a_2,a_3,a_4,a_5,a_6,a_7,a_8,a_9,a_{10})$ near $$(0,0,0,0,0,0,-\frac{3}{4}a_7,\frac{8}{7}a_{10},a_7,-\frac{16}{7}a_{10},0,a_{10}).$$
The proof is completed.
\end{proof}

\section{Appendix}
In this appendix, we give the expressions of $\tilde r_{ij}$ and $r_{ij}^{(1)},\ i,j\geq0$ by running the calculating programs in \cite{c15} through using Maple-17.  Formulas of $\bar a_{i,1},\ \bar b_{i,0},\ \alpha_{i,0},\ i=0,1,\dots,5$, $h_j,\ j=3,4,\dots,9$ and $\bar a_{0,2},\ \bar a_{1,2},\ \bar a_{0,3},\ \bar a_{1,3},\ \bar b_{0,1},\ \bar b_{1,1},\ \bar b_{0,2},\ \bar b_{1,2},\ \alpha_{0,1},\ \alpha_{1,1}$  are the same as these in \cite{c2}.
\begin{eqnarray*}
\tilde r_{00}&=&\alpha_{{0,0}}{\bar n}_{{0}}\\
\tilde r_{10}&=&\alpha_{{0,0}}{\bar{\mu}}_{{1}}{\bar n}_{{1}}+\alpha_{{1,0}}{\bar{\mu}}_{{1}}{\bar n}_{{0}}\\
\tilde r_{20}&=&\alpha_{{0,0}}{{\bar{\mu}}_{{1}}}^{2}{\bar n}_{{2}}+\alpha_{{1,0}}{{\bar{\mu}}_{{1}}}^{2}{\bar n}_{{1}}+\alpha_{{2,0}}{{\bar{\mu}}_{{1}}}^{2}{\bar n}_{{0}}+\alpha_{{0,0}}{\bar{\mu}}_{{2}}{\bar n}_{{1}}+\alpha_{{1,0}}{\bar{\mu}}_{{2}}{\bar n}_{{0}}\\
\tilde r_{30}&=&({{\bar{\mu}}_{{1}}}^{3}{\bar n}_{{3}}+2\,{\bar{\mu}}_{{1}}{\bar{\mu}}_{{2}}{\bar n}_{{2}}+{\bar{\mu}}_{{3}}{\bar n}_{{1}})\alpha_{{0,0}}
+({{\bar{\mu}}_{{1}}}^{3}{\bar n}_{{2}}+2\,{\bar{\mu}}_{{1}}{\bar{\mu}}_{{2}}{\bar n}_{{1}}+{\bar{\mu}}_{{3}}{\bar n}_{{0}})\alpha_{{1,0}}\\
&&+({{\bar{\mu}}_{{1}}
}^{3}{\bar n}_{{1}}+2\,{\bar{\mu}}_{{1}}{\bar{\mu}}_{{2}}{\bar n}_{{0}})\alpha_{{2,0}}
+{{\bar{\mu}}_{{1}}}^{3}{\bar n}_{{0}}\alpha_{{3,0}}\\
\tilde r_{i0}&=&\sum\limits_{l=0}^{(i)}m_{i0}^{(l)}\alpha_{l,0},\ \ \ \ i=4,6,7,8\\
\tilde r_{01}&=&{\bar n}_{{0}}\alpha_{{0,1}}\\
\tilde r_{11}&=&{\bar{\mu}}_{{1}}{\bar n}_{{1}}\alpha_{{0,1}}+{\bar{\mu}}_{{1}}{\bar n}_{{0}}\alpha_{{1,1}}\\
\tilde r_{21}&=&({{\bar{\mu}}_{{1}}}^{2}{\bar n}_{{2}}+{\bar{\mu}}_{{2}}{\bar n}_{{1}})\alpha_{{0,1}}+({{\bar{\mu}}_{{1}}}^{2}{\bar n}_{{1}}+{\bar{\mu}}_{{2}}{\bar n}_{{0}})\alpha_{{1,1}}+{{\bar{\mu}}_{{1}}}^{2}{\bar n}_{{0}}\alpha_{{2,1}}
\end{eqnarray*}
and
\begin{eqnarray*}
r_{01}^{(1)}&=&2\,{\bar n}_{{0}}\gamma_{{0,1}},\\
r_{11}^{(1)}&=&(2\,{{\bar{\mu}}_{{1}}}
^{2}{\bar n}_{{2}}+2\,{\bar{\mu}}_{{2}}{\bar n}_{{1}})\gamma_{{0,1}}+(2\,{{\bar{\mu}}_{{1}}}^{2}{\bar n}_{{1}}+2\,{\bar{\mu}}_{{2}}{\bar n}_{{0}
})\gamma_{{1,1}}+2\,{{\bar{\mu}}_{{1}}}^{2}{\bar n}_{{0}}\gamma_{{2,1}},\ \ \ \ \\
r_{i1}^{(1)}&=&\sum\limits_{l=0}^{2i}\tilde m_{i1}^{(l)}\gamma_{l,1},\ \ \ \ i=2,3,4,\\
r_{03}^{(1)}&=&2\,{\bar n}_{{0}}\gamma_{{0,3}}\\
r_{13}^{(1)}&=&(2\,{{\bar{\mu}}_{{1}}}
^{2}{\bar n}_{{2}}+2\,{\bar{\mu}}_{{2}}{\bar n}_{{1}})\gamma_{{0,3}}+(2\,{{\bar{\mu}}_{{1}}}^{2}{\bar n}_{{1}}+2\,{\bar{\mu}}_{{2}}{\bar n}_{{0}})\gamma_{{1,3}}+2\,{{\bar{\mu}}_{{1}}}^{2}{\bar n}_{{0}}\gamma_{{2,3}}
\end{eqnarray*}
where
\begin{eqnarray*}
m_{40}^{(0)}&=&{{\bar{\mu}}_{{1}}}^{4}{\bar n}_{{4}}+{\bar{\mu}}_{{4}}{\bar n}_{{1}}+2\,{\bar{\mu}}_{{1}}{\bar{\mu}}_{{3}}{\bar n}_{{2}}+3\,{{\bar{\mu}}_{{1}}}^{2}{\bar{\mu}}_{{2}}{\bar n}_{{3}}+{{\bar{\mu}}_{{2}}}^{2}{\bar n}_{{2}}\\
m_{40}^{(1)}&=&{{\bar{\mu}}_{{1}}}^{4}{\bar n}_{{3}}+3\,{{\bar{\mu}}_{{1}}}^{2}{\bar{\mu}}_{{2}}{\bar n}_{{2}}+2\,{\bar{\mu}}_{{1}}{\bar{\mu}}_{{3}}{\bar n}_{{1}}
+{{\bar{\mu}}_{{2}}}^{2}{\bar n}_{{1}}+{\bar{\mu}}_{{4}}{\bar n}_{{0}}\\
m_{40}^{(2)}&=&{{\bar{\mu}}_{{1}}}^{4}{\bar n}_{{2}}+3\,{{\bar{\mu}}_{{1}}}^{2}{\bar{\mu}}_{{2}}{\bar n}_{{1}}+2\,{\bar{\mu}}_{{1}}{\bar{\mu}}_{{3}}{\bar n}_{{0}}+{{\bar{\mu}}_{{2}}}^{2}{\bar n}_{{0}}\\
m_{40}^{(3)}&=&{{\bar{\mu}}_{{1}}}^{4}{\bar n}_{{1}}+3\,{{\bar{\mu}}_{{1}}}^{2}{\bar{\mu}}_{{2}}{\bar n}_{{0}}\\
m_{40}^{(4)}&=&{{\bar{\mu}}_{{1}}}^{4}{\bar n}_{{0}}\\
m_{60}^{(0)}&=&4\,{{\bar{\mu}}_{{1}}}^{3}{\bar{\mu}}_{{3}}{\bar n}_{{4}}
+6\,{{\bar{\mu}}_{{1}}}^{2}{{\bar{\mu}}_{{2}}}^{2}{\bar n}_{{4}}
+2\,{\bar{\mu}}_{{2}}{\bar{\mu}}_{{4}}{\bar n}_{{2}}+5\,{{\bar{\mu}}_{{1
}}}^{4}{\bar{\mu}}_{{2}}{\bar n}_{{5}}+2\,{\bar{\mu}}_{{1}}{\bar{\mu}}_{{5}}{\bar n}_{{2}}+3\,{{\bar{\mu}}_{{1}}}^{2}{\bar{\mu}}_{{4}}{\bar n}_{{3}}+{{\bar{\mu}}_{{2}}}^{3}{\bar n}_{{3}}\\
&&+{{\bar{\mu}}_{{3}}}^{2}{\bar n}_{{2}}
+{\bar{\mu}}_{{6}}{\bar n}_{{1}}+{{\bar{\mu}}_{{1}}}^{6}{\bar n}_{{6}}+6\,
{\bar{\mu}}_{{1}}{\bar{\mu}}_{{2}}{\bar{\mu}}_{{3}}{\bar n}_{{3}}\\
m_{60}^{(1)}&=&2\,{\bar{\mu}}_{{2}}{\bar{\mu}}_{{4}}{\bar n}_{{1}}+2\,{\bar{\mu}}_{{1}}{\bar{\mu}}_{{5}}{\bar n}_{{1}}+5\,{{\bar{\mu}}_{{1}}}^{4}{\bar{\mu}}_{{2}}{\bar n}_{{4}}
+4\,{{\bar{\mu}}_{{1}}}^{3}{\bar{\mu}}_{{3}}{\bar n}_{{3}}+6\,{{\bar{\mu}}_{{1}}}^{2}{{\bar{\mu}}_{{2}}}^{2}{\bar n}_{{3}}+3\,{{\bar{\mu}}_{{1}}}^{2}{\bar{\mu}}_{{4}}{\bar n}_{{2}}+{{\bar{\mu}}_{{1}}}^{6}{\bar n}_{{5}}\\
&&+{{\bar{\mu}}_{{2}}}^{3}{\bar n}_{{2}}+{{\bar{\mu}}_{{3}}}^{2}{\bar n}_{{1}}+{\bar{\mu}}_{{6}}{\bar n}_{{0}}+6\,{\bar{\mu}}_{{1}}{\bar{\mu}}_{{2}}{\bar{\mu}}_{{3}}{\bar n}_{{2}}\\
m_{60}^{(2)}&=&5\,{{\bar{\mu}}_{{1}}}^{4}{\bar{\mu}}_{{2}}{\bar n}_{{
3}}+2\,{\bar{\mu}}_{{2}}{\bar{\mu}}_{{4}}{\bar n}_{{0}}
+2\,{\bar{\mu}}_{{1}}{\bar{\mu}}_{{5}}{\bar n}_{{0}}+6\,{{\bar{\mu}}_{{1}}}^{2}{{\bar{\mu}}_{{2}}}^{2}{\bar n}_{{2}}+4\,{{\bar{\mu}}_{{1}}}^{3}{\bar{\mu}}_{{3}}{\bar n}_{{2}}+3\,{{\bar{\mu}}_{{1}}}^{2}{\bar{\mu}}_{{4}}{\bar n}_{{1}}+{{\bar{\mu}}_{{3}}}^{2}{\bar n}_{{0}}\\
&&+6\,{\bar{\mu}}_{{1}}{\bar{\mu}}_{{2}}{\bar{\mu}}_{{3}}{\bar n}_{{1}}+{{\bar{\mu}}_{{1}}}^{6}{\bar n}_{{4}}+{{\bar{\mu}}_{{2}}}^{3}{\bar n}_{{1}}\\
m_{60}^{(3)}&=&4\,{{\bar{\mu}}_{{1}}}^{3}{\bar{\mu}}_{{3}}{\bar n}_{{1}}
+5\,{{\bar{\mu}}_{{1}}}^{4}{\bar{\mu}}_{{2}}{\bar n}_{{2}}+6\,{{\bar{\mu}}_{{1}}}^{2}{{\bar{\mu}}_{{2}}}^{2}{\bar n}_{{1}}+3\,{{\bar{\mu}}_{{1}}}^{2}{\bar{\mu}}_{{4}}{\bar n}_{{0}}+6\,{\bar{\mu}}_{{1}}{\bar{\mu}}_{{2}}{\bar{\mu}}_{{3}}{\bar n}_{{0}}
+{{\bar{\mu}}_{{1}}}^{6}{\bar n
}_{{3}}+{{\bar{\mu}}_{{2}}}^{3}{\bar n}_{{0}}\\
m_{60}^{(4)}&=&5\,{{\bar{\mu}}_{{1}}}^{4}{\bar{\mu}}_{{
2}}{\bar n}_{{1}}+4\,{{\bar{\mu}}_{{1}}}^{3}{\bar{\mu}}_{{3}}{\bar n}_{{0}}
+6\,{{\bar{\mu}}_{{1}}}^{2}{{\bar{\mu}}_{{2}}}^{2}{\bar n}
_{{0}}+{{\bar{\mu}}_{{1}}}^{6}{\bar n}_{{2}}\\
m_{60}^{(5)}&=&5\,{{\bar{\mu}}_{{1}}}^{4}{\bar{\mu}}_{{2}}{\bar n}_{{0}}+{{\bar{\mu}}_{{1}}}^{6}{\bar n}_{{1}}\\
m_{60}^{(6)}&=&{{\bar{\mu}}_{{1}}}^{6}{\bar n}_{{0}}\\
m_{70}^{(0)}&=&2\,{\bar{\mu}}_{{2}}{\bar{\mu}}_{{5}}{\bar n}_{{2}}+2\,{\bar{\mu}}_{{3}}{\bar {\mu}}_{{4}}{\bar n}_{{2}}+6\,{\bar{\mu}}_{{1}}{\bar{\mu}}_{{2}}{\bar{\mu}}_{{4}}{\bar n}_{{3}}+3\,{{\bar{\mu}}_{{1}}}^{2}{\bar{\mu}}_{{5}}{\bar n}_{{3}}+4\,{\bar{\mu}}_{{1}}{{\bar{\mu}}_{{2}}}^{3}{\bar n}_{{4}}+4\,{{\bar{\mu}}_{{1}}}^{3}{\bar{\mu}}_{{4}}{\bar n}_{{4}}\\
&&+5\,{{\bar{\mu}}_{{1}}}^{4}{\bar{\mu}}_{{3}}{\bar n}_{{5}}+10\,{{\bar{\mu}}_{{1}}}^{3}{{\bar{\mu}}_{{2}}}^{2}{\bar n}_{{5}}+2\,{\bar{\mu}}_{{1}}{\bar{\mu}}_{{6}}{\bar n}_{{2}}+3\,{\bar{\mu}}_{{1}}{{\bar{\mu}}_{{3}}}^{2}{\bar n}_{{3}}+3\,{{\bar{\mu}}_{{2}}}^{2}{\bar{\mu}}_{{3}}{\bar n}_{{3}}+6\,{{\bar{\mu}}_{{1}}}^{5}{\bar{\mu}}_{{2}}{\bar n}_{{6}}\\
&&+\alpha_{{0,0}}{{\bar{\mu}}_{{1}}}^{7}{\bar n}_{{7}}+{\bar{\mu}}_{{7}}{\bar n}_{{1}}+12\,{{\bar{\mu}}_{{1}}}^{2}{\bar{\mu}}_{{2}}{\bar{\mu}}_{{3}}{\bar n}_{{4}}\\
m_{70}^{(1)}&=&2\,{\bar{\mu}}_{{2}}{\bar{\mu}}_{{5}}{\bar n}_{{1}}+6\,{{\bar{\mu}}_{{1}}}^{5}{\bar{\mu}}_{{2}}{\bar n}_{{5}}+5\,{{\bar{\mu}}_{{1}}}^{4}{\bar{\mu}}_{{3}}{\bar n}_{{4}}+2\,{\bar{\mu}}_{{1}}{\bar{\mu}}_{{6}
}{\bar n}_{{1}}+10\,{{\bar{\mu}}_{{1}}}^{3}{{\bar{\mu}}_{{2}}}^{2}{\bar n}_{{4}}+2\,{\bar{\mu}}_{{3}}{\bar{\mu}}_{{4}}{\bar n}_{{1}}\\
&&+4\,{\bar{\mu}}_{{1}}{{\bar{\mu}}_{{2}}}^{3}{\bar n}_{{3}}+4\,{{\bar{\mu}}_{{1}}}^{3}{\bar{\mu}}_{{4}}{\bar n}_{{3}}+3\,{{\bar{\mu}}_{{2}}}^{2}{\bar{\mu}}_{{
3}}{\bar n}_{{2}}+6\,{\bar{\mu}}_{{1}}{\bar{\mu}}_{{2}}{\bar{\mu}}_{{4}}{\bar n}_{{2}}+3\,{\bar{\mu}}_{{1}}{{\bar{\mu}}_{{3}}}^{2}{\bar n}_{{2}}+3\,{{\bar{\mu}}_{{1}}}^{2}{\bar{\mu}}_{{5}}{\bar n}_{{2}}\\
&&+{{\bar{\mu}}_{{1}}}^{7}{\bar n}_{{6}}+{\bar{\mu}}_{{7}}{\bar n}_{{0}}+12\,{{\bar{\mu}}_{{1}}}^{2}{\bar{\mu}}_{{2}}{\bar{\mu}}_{{3}}{\bar n}_{{3}}\\
m_{70}^{(2)}&=&4\,{{\bar{\mu}}_{{1}}}^{3}{\bar{\mu}}_{{4}}{\bar n}_{{2}}+2\,{\bar{\mu}}_{{1}}{\bar{\mu}}_{{6}}{\bar n}_{{0}}+10\,{{\bar{\mu}}_{{
1}}}^{3}{{\bar{\mu}}_{{2}}}^{2}{\bar n}_{{3}}+5\,{{\bar{\mu}}_{{1}}}^{4}{\bar{\mu}}_{{3}}{\bar n}_{{3}}+4\,{\bar{\mu}}_{{1}}{{\bar{\mu}}_{{2}}}^{3}{\bar n}_{{2}}+2\,{\bar{\mu}}_{{3}}{\bar{\mu}}_{{4}}{\bar n}_{{0}}\\
&&+3\,{\bar{\mu}}_{{1}}{{\bar{\mu}}_{{3}}}^{2}{\bar n}_{{1}}+3\,{{\bar{\mu}}_{{2}}}^{2}{\bar{\mu}}_{{3}}{\bar n}_{{1}}+2\,{\bar{\mu}}_{{2}}{\bar{\mu}}_{{5}}{\bar n}_{{0}}+6\,{{\bar{\mu}}_{{1}}}^{5}{\bar{\mu}}_{{2}}{\bar n}_{{4}}+{{\bar{\mu}}_{{1}}}^{7}{\bar n}_{{5}}+6\,{\bar{\mu}}_{{1}}{\bar{\mu}}_{{2}}{\bar{\mu}}_{{4}}{\bar n}_{{1}}\\
&&+3\,{{\bar{\mu}}_{{1}}}^{2}{\bar{\mu}}_{{5}}{
\bar n}_{{1}}+12\,{{\bar{\mu}}_{{1}}}^{2}{\bar{\mu}}_{{2}}{\bar{\mu}}_{{3}}{\bar n }_{{2}}\\
m_{70}^{(3)}&=&4\,{{\bar{\mu}}_{{1}}}^{3}{\bar{\mu}}_{{4}}{\bar n}_{{1}}+5\,{{\bar{\mu}}_{{1}}}^{4}{\bar{\mu}}_{{3}}{\bar n}_{{2}}+10\,{{\bar{\mu}}_{{1}}}^{3}{{\bar{\mu}}_{{2}}}^{2}{\bar n}_{{2}}+3\,{{\bar{\mu}}_{{2}}}^{2}{\bar{\mu}}_{{3}}{\bar n
}_{{0}}+3\,{{\bar{\mu}}_{{1}}}^{2}{\bar{\mu}}_{{5}}{\bar n}_{{0}}+3\,{\bar{\mu}}_{{1}}{{\bar{\mu}}_{{3}}}^{2}{\bar n}_{{0}}\\
&&+4\,{\bar{\mu}}_{{1}}{{\bar{\mu}}_{{2}}}^{3}{\bar n}_{{1}}+6\,{{\bar{\mu}}_{{1}}}^{5}{\bar{\mu}}_{{2}}{\bar n}_{{3}}+
{{\bar{\mu}}_{{1}}}^{7}{\bar n}_{{4}}+12\,{{\bar{\mu}}_{{1}}}^{2}{\bar{\mu}}_{{2}}{\bar{\mu}}_{{3}}{\bar n}_{{1}}+6\,{\bar{\mu}}_{{1}}{\bar{\mu}}_{{2}}{\bar{\mu}}_{{4}}{\bar n}_{{0}}\\
m_{70}^{(4)}&=&4\,{{\bar{\mu}}_{{1}}}^{3}{\bar{\mu}}_{{4}}{\bar n}_{{0}}+5\,{{\bar{\mu}}_{{1}}}^{4}{\bar{\mu}}_{{3}}{\bar n}_{{1}}+4\,{\bar{\mu}}_{{1}}{{\bar{\mu}}_{{2}}}^{3}{\bar n}_{{0}}+10\,{{\bar{\mu}}_{{1}}}^{3}{{\bar{\mu}}_{{2}}}^{2}{\bar n}_{{1}}+6\,{{\bar{\mu}}_{{1}}}^{5}{\bar{\mu}}_{{2}}{\bar n}_{{2}}+12\,{{\bar{\mu}}_{{1}}}^{2}{\bar{\mu}}_{{2}}{\bar{\mu}}_{{3}}{\bar n}_{{0}}\\
&&+{{\bar{\mu}}_{{1}}}^{7}{\bar n}_{{3}}\\
m_{70}^{(5)}&=&6\,{{\bar{\mu}}_{{1}}}^{5}{\bar{\mu}}_{{2}}{\bar n}_{{1}}+10\,{{\bar{\mu}}_{{1}}}^{3}{{\bar{\mu}}_{{2}}}^{2}{\bar n}_{{0}}+{{\bar{\mu}}_{{1}}}^{7}{\bar n}_{{2}}+5\,{{\bar{\mu}}_{{1}}}^{4}{\bar{\mu}}_{{3}}{\bar n}_{{0}}\\
m_{70}^{(6)}&=&6\,{{\bar{\mu}}_{{1}}}^{5}{\bar{\mu}}_{{2}}{\bar n}_{{0}}+{{\bar{\mu}}_{{1}}}^{7}{\bar n}_{{1}}\\
m_{70}^{(7)}&=&{{\bar{\mu}}_{{1}}}^{7}{\bar n}_{{0}}\\
m_{80}^{(0)}&=&2\,{\bar{\mu}}_{{2}}{\bar{\mu}}_{{6}}{\bar n}_{{2}}+2\,{\bar{\mu}}_{{1}}{\bar{\mu}}_{{7}}{\bar n}_{{2}}
+6\, {{\bar{\mu}}_{{1}}}^{5}{\bar{\mu}}_{{3}}{\bar n}_{{6}}+4\, {{\bar{\mu}}_{{1}}}^{3}{\bar{\mu}}_{{5}}{\bar n_4}+15\,{{\bar{\mu}}_{{1}}}^{4}{{\bar{\mu}}_{{2}}}^{2}{\bar n}_{{6}}+3\, {{\bar{\mu}}_{{1}}}^{2}{\bar{\mu}}_{{6}}{\bar n}_{{3}}\\
&&+2\, {\bar{\mu}}_{{3}}{\bar{\mu}}_{{5}}{\bar n}_{{2}}
+7\, {{\bar{\mu}}_{{1}}}^{6}{\bar{\mu}}_{{2}}{\bar n_7} +10\,{{\bar{\mu}}_{{1}}}^{2}{{\bar{\mu}}_{{2}}}^{3}{\bar n}_{{5}}+5\, {{\bar{\mu}}_{{1}}}^{4}{\bar{\mu}}_{{4}}{\bar n}_{{5}}+3\, {{\bar{\mu}}_{{2}}}^{2}{\bar{\mu}}_{{4}}{\bar n}_{{3}}
+3\, {\bar{\mu}}_{{2}}{{\bar{\mu}}_{{3}}}^{2}{
\bar n}_{{3}}\\
&&+ {\bar{\mu}}_{{8}}{\bar n}_{{1}}+ {{\bar{\mu}}_{{2}}}^{4}{\bar n}_{{4}}+ {{\bar{\mu}}_{{4}}}^{2}{\bar n}_{{2}}+{{\bar{\mu}}_{{1}}}^{8}{\bar n}_{{8}}+6\, {{\bar{\mu}}_{{1}}}^{2}{{\bar{\mu}}_{{3}}}^{2}{\bar n}_{{4}}+12\, {\bar{\mu}}_{{1}}{{\bar{\mu}}_{{2}}}^{2}{\bar{\mu}}_{{3}}{\bar n}_{{4}}\\
&&+6\, {\bar{\mu}}_{{1}}{\bar{\mu}}_{{3}}{\bar{\mu}}_{{4}}
{\bar n}_{{3}}+6\, {\bar{\mu}}_{{1}}{\bar{\mu}}_{{2}}{\bar{\mu}}_{{5}}{\bar n}_{{3}}+20\, {{\bar{\mu}}_{{1}}}^{3}{\bar{\mu}}_{{2}}{\bar{\mu}}_{{3}}{\bar n}_{{5}}+12\, {{\bar{\mu}}_{{1}}}^{2}{\bar{\mu}}_{{2}}{\bar{\mu}}_{{4}}{\bar n}_{{4}}\\
m_{80}^{(1)}&=&15\,{{\bar{\mu}}_{{1}}}^{4}{{\bar{\mu}}_{{2}}}^{2}{ \bar n}_{{5}}+6\, {{\bar{\mu}}_{{1}}}^{2}{{\bar{\mu}}_{{3}}}^{2}{\bar n}_{{3}}+7\, {{\bar{\mu}}_{{1}}}^{6}{\bar{\mu}}_{{2}}{\bar n}_{{6}}+10\, {{\bar{\mu}}_{{1}}}^{2}{{\bar{\mu}}_{{2}}}^{3}{\bar n}_{{4}}+3\, {{\bar{\mu}}_{{1}}}^{2}{\bar{\mu}}_{{6}}{\bar n}_{{2}}+5\, {{\bar{\mu}}_{{1}}}^{4}{\bar{\mu}
}_{{4}}{\bar n}_{{4}}\\
&&+3\, {\bar{\mu}}_{{2}}{{\bar{\mu}}_{{3}}}^{2}{\bar n}_{{2}}+2\, {\bar{\mu}}_{{2}}{\bar{\mu}}_{{6}}{\bar n}_{{1}}+2\, {\bar{\mu}}_{{1}}{\bar{\mu}}_{{7}}{\bar n}_{{1}}+3\, {{\bar{\mu}}_{{2}}}^{2}{\bar{\mu}}_{{4}}{\bar n}_{{2}}+2\, {\bar{\mu}}_{{3}}{\bar{\mu}}_{{5}}{\bar n}_{{1}}+ {\bar{\mu}}_{{8}}{\bar n}_{{0}}\\
&&+ {{\bar{\mu}}_{{2}}}^{4}{\bar n}_{{3}}+ {{\bar{\mu}}_{{4}}}^{2}{\bar n}_{{1}}+ {{\bar{\mu}}_{{1}}}^{8}{\bar n}_{{7}}+4\, {{\bar{\mu}}_{{1}}}^{3}{\bar{\mu}}_{{5}}{\bar n}_{{3}}+6\, {\bar{\mu}}_{{1}}{\bar{\mu}}_{{2}
}{\bar{\mu}}_{{5}}{\bar n}_{{2}}+6\, {\bar{\mu}}_{{1}}{
\bar{\mu}}_{{3}}{\bar{\mu}}_{{4}}{\bar n}_{{2}}\\
&&+20\, {{\bar{\mu}}_{{1}}}^{3}{\bar{\mu}}_{{2}}{\bar{\mu}}_{{3}}{\bar n}_{{4}}+12\, {\bar{\mu}}_{{1}}{{\bar{\mu}}_{{2}}}^{2}{\bar{\mu}}_{{3}}{\bar n}_{{3}}+6\, {{\bar{\mu}}_{{1}}}^{5}{\bar{\mu}}_{{3}}{\bar n}_{{5}}
+12\, {{\bar{\mu}}_{{1}}}^{2}{\bar{\mu}}_{{2}}{\bar{\mu}}_{{4}}{\bar n}_{{3}}\\
m_{80}^{(2)}&=&2\,{\bar{\mu}}_{{3}}{\bar{\mu}}_{{5}}{\bar n}_{{0}}+6\, {{\bar{\mu}}_{{1}}}^{2}{{\bar{\mu}}_{{3}}}^{2}{\bar n}_{{2}}+6\, {{\bar{\mu}}_{{1}}}^{5}{\bar{\mu}}_{{3}}{\bar n}_{{4}}+15\, {{\bar{\mu}}_{{1}}}^{4}{{\bar{\mu}}_{{2}}}^{2}{\bar n}_{{4}}+
2\, {\bar{\mu}}_{{2}}{\bar{\mu}}_{{6}}{\bar n}_{{0}}+3\, {\bar{\mu}}_{{2}}{{\bar{\mu}}_
{{3}}}^{2}{\bar n}_{{1}}\\
&&+3\, {{\bar{\mu}}_{{1}}}^{2}{\bar{\mu}}_{{6}}{\bar n}_{{1}}+3\, {{\bar{\mu}}_{{2}}}^{2}{\bar{\mu}}_{{4}}{\bar n}_{{1}}+7\, {{\bar{\mu}}_{{1}}}^{6}{\bar{\mu}}_{{2}}{\bar n}_{{5}}+5\, {{\bar{\mu}}_{{1}}}^{4}{\bar{\mu}}_{{4}}{\bar n}_{{3}}+10\, {{\bar{\mu}}_{{1}}}^{2}{{\bar{\mu}}_{{2}}}^{3}{\bar n}_{{3}}+4\, {{\bar{\mu}}_{{1}}}^{3}{\bar{\mu}}_{{5}}{\bar n}_{{2}}\\
&&+2\, {\bar{\mu}}_{{1}}{\bar{\mu}}_{{7}}{\bar n}_{{0}}+ {{\bar{\mu}}_{{4}}}^{2}{\bar n}_{{0}}+ {{
\bar{\mu}}_{{1}}}^{8}{\bar n}_{{6}}+ {{\bar{\mu}}_{{2}}}^{4}{\bar n}_{{2}}+6\, {\bar{\mu}}_{{1}}{\bar{\mu}}_{{2}}{\bar{\mu}}_{{5}}{\bar n}_{{1}}+6\, {\bar{\mu}}_{{1}}{\bar{\mu}}_{{3}}{\bar{\mu}}_{{4}}{\bar n}_{{1}}\\
&&+20\, {{\bar{\mu}}_{{1}}}^{3}{\bar{\mu}}_{{2}}{\bar{\mu}}_{{3}}{\bar n}_{{3}}+12\, {\bar{\mu}}_{{1}}{{\bar{\mu}}_{{2}}}^{2}{\bar{\mu}}_{{3}}{\bar n}_{{2}}+12\, {{\bar{\mu}}_{{1}}}^{2}{\bar{\mu}}_{{2}}{\bar{\mu}}_{{4}}{\bar n}_{{2}}\\
m_{80}^{(3)}&=&6\,{{\bar{\mu}}_{{1}}}^{2}{{\bar{\mu}}_{{3}}}^{2}{\bar n}_{{1}}+4\, {{\bar{\mu}}_{{1}}}^{3}{\bar{\mu}}_{{5}}{
\bar n}_{{1}}+15\, {{\bar{\mu}}_{{1}}}^{4}{{\bar{\mu}}_{{2}}}^{2}{\bar n}_{{3}}+6\, {{\bar{\mu}}_{{1}}}^{5}{\bar{\mu}}_{{3}}{\bar n}_{{3}}+5\, {{\bar{\mu}}_{{1}}}^{4}{\bar{\mu}}_{{4}}{\bar n}_{{2}}+3\, {\bar{\mu}}_{{2}}{{\bar{\mu}}_{{3}}}^{2}{\bar n}_{{0}}\\
&&+7\, {{\bar{\mu}}_{{1}}}^{6}{\bar{\mu}}_{{2}}{\bar n}_{{4}}+3\, {{\bar{\mu}}_
{{2}}}^{2}{\bar{\mu}}_{{4}}{\bar n}_{{0}}+10\, {{\bar{\mu}}_{{1}}}^{2}{{\bar{\mu}}_{{2}}}^{3}{\bar n}_{{2}}+3\, {{\bar{\mu}}_{{1}}}^{2}{\bar{\mu}}_{{6}}{\bar n}_{{0}}+ {{\bar{\mu}}_{{1}}}^{8}{\bar n}_{{5}}+6\, {\bar{\mu}}_{{1}}{\bar{\mu}}_{{2}}{\bar{\mu}}_{{5}}{\bar n}_{{0}}\\
&&+ {{\bar{\mu}}_{{2}}}^{4}{\bar n}_{{1}}+20\, {{\bar{\mu}}_{{1}}}^{3}{\bar{\mu}}_{{2}}{\bar{\mu}}_{{3}}{\bar n}_{{2}}+12\, {\bar{\mu}}_{{1}}{{\bar{\mu}}_{{2}}}^{2}{\bar{\mu}}_{{3}}{\bar n}_{{1}}+12\, {{\bar{\mu}}_{{1}}}^{2}{\bar{\mu}}_{{2}}{\bar{\mu}}_{{4}}{\bar n}_{{1}}+6\, {\bar{\mu}}_{{1}}{\bar{\mu}}_{{3}}{\bar{\mu}}_{{4}}{\bar n}_{{0}}\\
m_{80}^{(4)}&=&7\,{{\bar{\mu}}_{{1}}}^{6}{\bar{\mu}}_{{2}}{\bar n}_{{3}}+6\, {{\bar{\mu}}_{{1}}}^{5}{\bar{\mu}}_{{3}}{\bar n}_{{2}}+15\, {{\bar{\mu}}_{{1}}}^{4}{{\bar{\mu}}_{{2}}}^{2}{\bar n}_{{2}}+10\, {{\bar{\mu}}_{{1}}}^{2}{{\bar{\mu}}_{{2}}}^{3}{
\bar n}_{{1}}+5\,\alpha_{{4
,0}}{{\bar{\mu}}_{{1}}}^{4}{\bar{\mu}}_{{4}}{\bar n}_{{1}}+4\, {{\bar{\mu}}_{{1}}}^{3}{\bar{\mu}}_{{5}}{\bar n}_{{0}}\\
&&+6\, {{\bar{\mu}}_{{1}}}^{2}{{\bar{\mu}}_{{3}}}^{2}{\bar n}_{{0}}+ {{\bar{\mu}}_{{1}}}^{8}{\bar n}_
{{4}}+ {{\bar{\mu}}_{{2}}}^{4}{\bar n}_{{0}}+20\, {{\bar{\mu}}_{{1}}}^{3}{\bar{\mu}}_{{2}}{\bar{\mu}}_{{3}}{\bar n}_{{1}}+12\, {{\bar{\mu}}_{{1}}}^{2}{\bar{\mu}}_{{2}}{\bar{\mu}}_{{4}}{\bar n}_{{0}}
+12\, {\bar{\mu}}_{{1}}{{\bar{\mu}}_{{2}}}^{2}{\bar{\mu}}_{{3}}{\bar n}_{{0}}\\
m_{80}^{(5)}&=&10\,{{\bar{\mu}}_{{1}}}^{2}{{\bar{\mu}}_{{2}}}^{3}{\bar n}_{{0}}+15\, {{\bar{\mu}}_{{1}}}^{4}{{\bar{\mu}}_{{2}}}^{2}{\bar n}_{{1}}+5\, {{\bar{\mu}}_{{1}}}^{4}{\bar{\mu}}_{{4}}{\bar n}_{{0}}+6\, {{\bar{\mu}}_{{1}}}^{5}{\bar{\mu}}_{{3}}{\bar n}_{{1}}+7\, {{\bar{\mu}}_{{1}}}^{6}{\bar{\mu}}_{{2}}{\bar n}_{{2}}+20\, {{\bar{\mu}}_{{1}}}^{3}{\bar{\mu}}_{{2}}{\bar{\mu}}_{{3}}{\bar n}_{{0}}\\
&&+ {{\bar{\mu}}_{{1}}}^{8}{\bar n}_{{3}}\\
m_{80}^{(6)}&=&15\,{{\bar{\mu}}_{{1}}}^{4}{{\bar{\mu}}_{{2}}}^{2}{\bar n}_{{0}}+6\, {{\bar{\mu}}_{{1}}}^{5}{\bar{\mu}}_{{3}}{\bar n}_{{0}}+7\, {{\bar{\mu}}_{{1}}}^{6}{\bar{\mu}}_{{2}}{\bar n}_{{1}}+ {{
\bar{\mu}}_{{1}}}^{8}{\bar n}_{{2}}\\
m_{80}^{(7)}&=&7\,{{\bar{\mu}}_{{1}}}^{6}{\bar{\mu}}_{{2}}{\bar n}_{{0}}
+{{\bar{\mu}}_{{1}}}^{8}{\bar n}_{{1}}\\
m_{80}^{(8)}&=&{{\bar{\mu}}_{{1}}}^{8}{\bar n}_{{0}}
\end{eqnarray*}
and
\begin{eqnarray*}
\tilde m_{21}^{(0)}&=&4\,{\bar{\mu}}_{{1}}{\bar{\mu}}_{{3}}{\bar n}_{{2}}+6\,{{\bar{\mu}}_{{1}}}^{2}{\bar{\mu}}_{{2}}{\bar n}_{{3}}+2\,{{\bar{\mu}}_{{2}}}^{2}{\bar n}_{{2}}+2\,{\bar{\mu}}_{{4}}{\bar n}_{{1}}+2\,{{\bar{\mu}}_{{1}}}^{4}{\bar n}_{{4}}\\
\tilde m_{21}^{(1)}&=&2\,{{\bar{\mu}}_{{1}}}^{4}{\bar n}_{{3}}+6\,{{\bar{\mu}}_{{1}}}^{2}{\bar{\mu}}_{{2}}{\bar n}_{{2}}
+4\,{\bar{\mu}}_{{1}}{\bar{\mu}}_{{3}}{\bar n}_{{1}}+2\,{{\bar{\mu}}_{{2}}}^{2}{\bar n}_{{1}}
+2\,{\bar{\mu}}_{{4}}{\bar n}_{{0}}\\
\tilde m_{21}^{(2)}&=&6\,{{\bar{\mu}}_{{1}}}^{2}{\bar{\mu}}_{{2}}{\bar n}_{{1}}+2\,{{\bar{\mu}}_{{1}}}^{4}{\bar n}_{{2}}+4\,{\bar{\mu}}_{{1}}{\bar{\mu}}_{{3}}{\bar n}_{{0}}+2\,{{\bar{\mu}}_{{2}}}^{2}{\bar n}_{{0}}\\
\tilde m_{21}^{(3)}&=&6\,{{\bar{\mu}}_{{1}}}^{2}{\bar{\mu}}_{{2}}{\bar n}_{{0}}+2\,{{\bar{\mu}}_{{1}}}^{4}{\bar n}_{{1}}\\
\tilde m_{21}^{(4)}&=&2\,{{\bar{\mu}}_{{1}}}^{4}{\bar n}_{{0}}\\
\tilde m_{31}^{(0)}&=&4\,{\bar{\mu}}_{{2}}{\bar{\mu}}_{{4}}{\bar n}_{{2}}+4\,{\bar{\mu}}_{{1}}{\bar{\mu}}_{{5}}{\bar n}_{{2}}+10\,{{\bar{\mu}}_{{1}}}^{4}{\bar{\mu}}_{{2}}{\bar n}_{{5}}+8\,{{\bar{\mu}}_{{1}}}^{3}{\bar{\mu}}_{{3}}{\bar n}_{{4}}+12\,{{\bar{\mu}}_{{1}}}^{2}{{\bar{\mu}}_{{2}}}^{2}{\bar n}_{{4}}+2\,{\bar{\mu}}_{{6}}{\bar n}_{{1}}+2\,{{\bar{\mu}}_{{3}}}^{2}{\bar n}_{{2}}\\
&&+6\,{{\bar{\mu}}_{{1}}}^{2}{\bar{\mu}}_{{4}}{\bar n}_{{3}}+2\,{{\bar{\mu}}_{{2}}}^{3}{\bar n}_{{3}}+12\,{\bar{\mu}}_{{1}}{\bar{\mu}}_{{2}}{\bar{\mu}}_{{3}}{\bar n}_{{3}}+2\,{{\bar{\mu}}_{{1}}}^{6}{\bar n}_{{6}}\\
\tilde m_{31}^{(1)}&=&10\,{{\bar{\mu}}_{{1}}}^{4}{\bar{\mu}}_{{2}}{\bar n}_{{4}}+4\,{\bar{\mu}}_{{2}}{\bar{\mu}}_{{4}}{\bar n}_{{1}}+4\,{\bar{\mu}}_{{1}}{\bar{\mu}}_{{5}}{\bar n}_{{1}}+12\,{{\bar{\mu}}_{{1}}}^{2}{{\bar{\mu}}_{{2}}}^{2}{\bar n}_{{3}}+8\,{{\bar{\mu}}_{{1}}}^{3}{\bar{\mu}}_{{3}}{\bar n}_{{3}}+6\,{{\bar{\mu}}_{{1}}}^{2}{\bar{\mu}}_{{4}}{\bar n}_{{2}}\\
&&+12\,{\bar{\mu}}_{{1}}{\bar{\mu}}_{{2}}{\bar{\mu}}_{{3}}{\bar n}_{{2}}
+2\,{{\bar{\mu}}_{{1}}}^{6}{\bar n}_{{5}}+2\,{\bar{\mu}}_{{6}}{\bar n}_{{0}}+2\,{{\bar{\mu}}_{{2}}}^{3}{\bar n}_{{2}}+2\,{{\bar{\mu}}_{{3}}}^{2}{\bar n}_{{1}}\\
\tilde m_{31}^{(2)}&=&8\,{{\bar{\mu}}_{{1}}}^{3}{\bar{\mu}}_{{3}}{\bar n}_{{2}}
+4\,{\bar{\mu}}_{{1}}{\bar{\mu}}_{{5}}{\bar n}_{{0}}+12\,{{\bar{\mu}}_{{1}}}^{2}{{\bar{\mu}}_{{2}}}^{2}{\bar n}_{{2}}+2\,{{\bar{\mu}}_{{1}}}^{6}{\bar n}_{{4}}+6\,{{\bar{\mu}}_{{1}}}^{2}{\bar{\mu}}_{{4}}{\bar n}_{{1}}+2\,{{\bar{\mu}}_{{2}}}^{3}{\bar n}_{{1}}+2\,{{\bar{\mu}}_{{3}}}^{2}{\bar n}_{{0}}\\
&&+4\,{\bar{\mu}}_{{2}}{\bar{\mu}}_{{4}}{\bar n}_{{0}}+10\,{{\bar{\mu}}_{{1}}}^{4}{\bar{\mu}}_{{2}}{\bar n}_{{3}}+12\,{\bar{\mu}}_{{1}}{\bar{\mu}}_{{2}}{\bar{\mu}}_{{3}}{\bar n}_{{1}}\\
\tilde m_{31}^{(3)}&=&6\,{{\bar{\mu}}_{{1}}}^{2}{\bar{\mu}}_{{4}}{\bar n}_{{0}}+12\,{\bar{\mu}}_{{1}}{\bar{\mu}}_{{2}}{\bar{\mu}}_{{3}}{\bar n}_{{0}}
+8\,{{\bar{\mu}}_{{1}}}^{3}{\bar{\mu}}_{{3}}{\bar n}_{{1}}+12\,{{\bar{\mu}}_{{1}}}^{2}{{\bar{\mu}}_{{2}}}^{2}{\bar n}_{{1}}+10\,{{\bar{\mu}}_{{1}}}^{4}{\bar{\mu}}_{{2}}{\bar n}_{{2}}+2\,{{\bar{\mu}}_{{1}}}^{6}{\bar n}_{{3}}\\
&&+2\,{{\bar{\mu}}_{{2}}}^{3}{\bar n}_{{0}}\\
\tilde m_{31}^{(4)}&=&10\,{{\bar{\mu}}_{{1}}}^{4}{\bar{\mu}}_{{2}}{\bar n}_{{1}}+8\,{{\bar{\mu}}_{{1}}}^{3}{\bar{\mu}}_{{3}}{\bar n}_{{0}}
+12\,{{\bar{\mu}}_{{1}}}^{2}{{\bar{\mu}}_{{2}}}^{2}{\bar n}_{{0}}+2\,{{\bar{\mu}}_{{1}}}^{6}{\bar n}_{{2}}\\
\tilde m_{31}^{(5)}&=&2\,{{\bar{\mu}}_{{1}}}^{6}{\bar n}_{{1}}+10\,{{\bar{\mu}}_{{1}}}^{4}{\bar{\mu}}_{{2}}{\bar n}_{{0}}\\
\tilde m_{31}^{(6)}&=&2\,{{\bar{\mu}}_{{1}}}^{6}{\bar n}_{{0}}\\
\tilde m_{41}^{(0)}&=&2\,{{\bar{\mu}}_{{1}}}^{8}{\bar n}_{{8}}+4\,{\bar{\mu}}_{{1}}{\bar{\mu}}_{{7}}{\bar n}_{{2}}+6\,{\bar{\mu}}_{{2}}{{\bar{\mu}}_{{3}}}^{2}{\bar n}_{{3}}
+14\,{{\bar{\mu}}_{{1}}^{6}}{\bar{\mu}}_{{2}}{\bar n}_{{7}}+12\,{{\bar{\mu}}_{{1}}}^{5}{\bar{\mu}}_{{3}}{\bar n}_{{6}}+12\,{{\bar{\mu}}_{{1}}}^{2}{{\bar{\mu}}_{{3}}}^{2}{\bar n}_{{4}}+2\,{{\bar{\mu}}_{{2}}}^{4}{\bar n}_{{4}}\\
&&+30\,{{\bar{\mu}}_{{1}}}^{4}{{\bar{\mu}}_{{2}}}^{2}{\bar n}_{{6}}+8\,{{\bar{\mu}}_{{1}}}^{3}{\bar{\mu}}_{{5}}{\bar n}_{{4}}+6\,{{\bar{\mu}}_{{2}}}^{2}{\bar{\mu}}_{{4}}{\bar n}_{{3}}+6\,{{\bar{\mu}}_{{1}}}^{2}{\bar{\mu}}_{{6}}{\bar n}_{{3}}+10\,{{\bar{\mu}}_{{1}}}^{4}{\bar{\mu}}_{{4}}{\bar n}_{{5}}+2\,{{\bar{\mu}}_{{4}}}^{2}{\bar n}_{{2}}+2\,{\bar{\mu}}_{{8}}{\bar n}_{{1}}\\
&&+24\,{\bar{\mu}}_{{1}}{{\bar{\mu}}_{{2}}}^{2}{\bar{\mu}}_{{3}}{\bar n}_{{4}}+12\,{\bar{\mu}}_{{1}}{\bar{\mu}}_{{3}}{\bar{\mu}}_{{4}}{\bar n}_{{3}}+24\,{{\bar{\mu}}_{{1}}}^{2}{\bar{\mu}}_{{2}}{\bar{\mu}}_{{4}}{\bar n}_{{4}}+40\,{{\bar{\mu}}_{{1}}}^{3}{\bar{\mu}}_{{2}}{\bar{\mu}}_{{3}}{\bar n}_{{5}}+12\,{\bar{\mu}}_{{1}}{\bar{\mu}}_{{2}}{\bar{\mu}}_{{5}}{\bar n}_{{3}}\\
&&
+4\,{\bar{\mu}}_{{2}}{\bar{\mu}}_{{6}}{\bar n}_{{2}}+4\,{\bar{\mu}}_{{3}}{\bar{\mu}}_{{5}}{\bar n}_{{2}}+20\,{{\bar{\mu}}_{{1}}}^{2}{{\bar{\mu}}_{{2}}}^{3}{\bar n}_{{5}}\\
\tilde m_{41}^{(1)}&=&4\,{\bar{\mu}}_{{2}}{\bar{\mu}}_{{6}}{\bar n}_{{1}}+2\,{{\bar{\mu}}_{{1}}}^{8}{\bar n}_{{7}}+12\,{\bar{\mu}}_{{1}}{\bar{\mu}}_{{3}}{\bar{\mu}}_{{4}}{\bar n}_{{2}}+24\,{{\bar{\mu}}_{{1}}}^{2}{\bar{\mu}}_{{2}}{\bar{\mu}}_{{4}}{\bar n}_{{3}}+4\,{\bar{\mu}}_{{3}}{\bar{\mu}}_{{5}}{\bar n}_{{1}}
+12\,{\bar{\mu}}_{{1}}{\bar{\mu}}_{{2}}{\bar{\mu}}_{{5}}{\bar n}_{{2}}\\
&&+40\,{{\bar{\mu}}_{{1}}}^{3}{\bar{\mu}}_{{2}}{\bar{\mu}}_{{3}}{\bar n}_{{4}}+24\,{\bar{\mu}}_{{1}}{{\bar{\mu}}_{{2}}}^{2}{\bar{\mu}}_{{3}}{\bar n}_{{3}}+6\,{\bar{\mu}}_{{2}}{{\bar{\mu}}_{{3}}}^{2}{\bar n}_{{2}}+4\,{\bar{\mu}}_{{1}}{\bar{\mu}}_{{7}}{\bar n}_{{1}}+20\,{{\bar{\mu}}_{{1}}}^{2}{{\bar{\mu}}_{{2}}}^{3}{\bar n}_{{4}}+14\,{{\bar{\mu}}_{{1}}}^{6}{\bar{\mu}}_{{2}}{\bar n}_{{6}}\\
&&+12\,{{\bar{\mu}}_{{1}}}^{2}{{\bar{\mu}}_{{3}}}^{2}{\bar n}_{{3}}+12\,{{\bar{\mu}}_{{1}}}^{5}{\bar{\mu}}_{{3}}
{\bar n}_{{5}}+8\,{{\bar{\mu}}_{{1}}}^{3}{\bar{\mu}}_{{5}}{\bar n}_{{3}}+2\,{{\bar{\mu}}_{{2}}}^{4}{\bar n}_{{3}}+2\,{{\bar{\mu}}_{{4}}}^{2}{\bar n}_{{1}}+10\,{{\bar{\mu}}_{{1}}}^{4}{\bar{\mu}}_{{4}}{\bar n}_{{4}}+2\,{\bar{\mu}}_{{8}}{\bar n}_{{0}}\\
&&+30\,{{\bar{\mu}}_{{1}}}^{4}{{\bar{\mu}}_{{2}}}^{2}{\bar n}_{{5}}+6\,{{\bar{\mu}}_{{2}}}^{2}{\bar{\mu}}_{{4}}{\bar n}_{{2}}+6\,{{\bar{\mu}}_{{1}}}^{2}{\bar{\mu}}_{{6}}{\bar n}_{{2}}\\
\tilde m_{41}^{(2)}&=&2\,{{\bar{\mu}}_{{4}}}^{2}{\bar n}_{{0}}+24\,{{\bar{\mu}}_{{1}}}^{2}{\bar{\mu}}_{{2}}{\bar{\mu}}_{{4}}{\bar n}_{{2}}+40\,{{\bar{\mu}}_{{1}}}^{3}{\bar{\mu}}_{{2}}{\bar{\mu}}_{{3}}{\bar n}_{{3}}+14\,{{\bar{\mu}}_{{1}}}^{6}{\bar{\mu}}_{{2}}{\bar n}_{{5}}+6\,{{\bar{\mu}}_{{1}}}^{2}{\bar{\mu}}_{{6}}{\bar n}_{{1}}+8\,{{\bar{\mu}}_{{1}}}^{3}{\bar{\mu}}_{{5}}{\bar n}_{{2}}\\
&&+12\,{{\bar{\mu}}_{{1}}}^{5}{\bar{\mu}}_{{3}}{\bar n}_{{4}}+12\,{{\bar{\mu}}_{{1}}}^{2}{{\bar{\mu}}_{{3}}}^{2}{\bar n}_{{2}}
+30\,{{\bar{\mu}}_{{1}}}^{4}{{\bar{\mu}}_{{2}}}^{2}{\bar n}_{{4}}+4\,{\bar{\mu}}_{{1}}{\bar{\mu}}_{{7}}{\bar n}_{{0}}+2\,{{\bar{\mu}}_{{2}}}^{4}{\bar n}_{{2}}+10\,{{\bar{\mu}}_{{1}}}^{4}{\bar{\mu}}_{{4}}{\bar n}_{{3}}\\
&&+2\,{{\bar{\mu}}_{{1}}}^{8}{\bar n}_{{6}}+12\,{\bar{\mu}}_{{1}}{\bar{\mu}}_{{2}}{\bar{\mu}}_{{5}}{\bar n}_{{1}}+12\,{\bar{\mu}}_{{1}}{\bar{\mu}}_{{3}}{\bar{\mu}}_{{4}}{\bar n}_{{1}}+6\,{{\bar{\mu}}_{{2}}}^{2}{\bar{\mu}}_{{4}}{\bar n}_{{1}}+20\,{{\bar{\mu}}_{{1}}}^{2}{{\bar{\mu}}_{{2}}}^{3}{\bar n}_{{3}}+4\,{\bar{\mu}}_{{3}}{\bar{\mu}}_{{5}}{\bar n}_{{0}}\\
&&+24\,{\bar{\mu}}_{{1}}{{\bar{\mu}}_{{2}}}^{2}{\bar{\mu}}_{{3}}{\bar n}_{{2}}+6\,{\bar{\mu}}_{{2}}{{\bar{\mu}}_{{3}}}^{2}{\bar n}_{{1}}
+4\,{\bar{\mu}}_{{2}}{\bar{\mu}}_{{6}}{\bar n}_{{0}}\\
\tilde m_{41}^{(3)}&=&+24\,{\bar{\mu}}_{{1}}{{\bar{\mu}}_{{2}}}^{2}{\bar{\mu}}_{{3}}{\bar n}_{{1}}
+12\,{\bar{\mu}}_{{1}}{\bar{\mu}}_{{3}}{\bar{\mu}}_{{4}}{\bar n}_{{0}}+12\,{\bar{\mu}}_{{1}}{\bar{\mu}}_{{2}}{\bar{\mu}}_{{5}}{\bar n}_{{0}}+24\,{{\bar{\mu}}_{{1}}}^{2}{\bar{\mu}}_{{2}}{\bar{\mu}}_{{4}}{\bar n}_{{1}}+40\,{{\bar{\mu}}_{{1}}}^{3}{\bar{\mu}}_{{2}}{\bar{\mu}}_{{3}}{\bar n}_{{2}}\\
&&+12\,{{\bar{\mu}}_{{1}}}^{5}{\bar{\mu}}_{{3}}{\bar n}_{{3}}+20\,{{\bar{\mu}}_{{1}}}^{2}{{\bar{\mu}}_{{2}}}^{3}{\bar n}_{{2}}
+14\,{{\bar{\mu}}_{{1}}}^{6}{\bar{\mu}}_{{2}}{\bar n}_{{4}}+6\,{{\bar{\mu}}_{{1}}}^{2}{\bar{\mu}}_{{6}}{\bar n}_{{0}}+6\,{\bar{\mu}}_{{2}}{{\bar{\mu}}_{{3}}}^{2}{\bar n}_{{0}}
+8\,{{\bar{\mu}}_{{1}}}^{3}{\bar{\mu}}_{{5}}{\bar n}_{{1}}\\
&&+30\,{{\bar{\mu}}_{{1}}}^{4}{{\bar{\mu}}_{{2}}}^{2}{\bar n}_{{3}}+12\,{{\bar{\mu}}_{{1}}}^{2}{{\bar{\mu}}_{{3}}}^{2}{\bar n}_{{1}}
+6\,{{\bar{\mu}}_{{2}}}^{2}{\bar{\mu}}_{{4}}{\bar n}_{{0}}+10\,{{\bar{\mu}}_{{1}}}^{4}{\bar{\mu}}_{{4}}{\bar n}_{{2}}+2\,{{\bar{\mu}}_{{1}}}^{8}{\bar n}_{{5}}2\,{{\bar{\mu}}_{{2}}}^{4}{\bar n}_{{1}}\\
\tilde m_{41}^{(4)}&=&2\,{{\bar{\mu}}_{{2}}}^{4}{\bar n}_{{0}}+24\,{\bar{\mu}}_{{1}}{{\bar{\mu}}_{{2}}}^{2}{\bar{\mu}}_{{3}}{\bar n}_{{0}}+24\,{{\bar{\mu}}_{{1}}}^{2}{\bar{\mu}}_{{2}}{\bar{\mu}}_{{4}}{\bar n}_{{0}}+40\,{{\bar{\mu}}_{{1}}}^{3}{\bar{\mu}}_{{2}}{\bar{\mu}}_{{3}}{\bar n}_{{1}}+14\,{{\bar{\mu}}_{{1}}}^{6}{\bar{\mu}}_{{2}}{\bar n}_{{3}}+12\,{{\bar{\mu}}_{{1}}}^{2}{{\bar{\mu}}_{{3}}}^{2}{\bar n}_{{0}}\\
&&+30\,{{\bar{\mu}}_{{1}}}^{4}{{\bar{\mu}}_{{2}}}^{2}{\bar n}_{{2}}+8\,{{\bar{\mu}}_{{1}}}^{3}{\bar{\mu}}_{{5}}{\bar n}_{{0}}+10\,{{\bar{\mu}}_{{1}}}^{4}{\bar{\mu}}_{{4}}{\bar n}_{{1}}+2\,{{\bar{\mu}}_{{1}}}^{8}{\bar n}_{{4}}+20\,{{\bar{\mu}}_{{1}}}^{2}{{\bar{\mu}}_{{2}}}^{3}{\bar n}_{{1}}+12\,{{\bar{\mu}}_{{1}}}^{5}{\bar{\mu}}_{{3}}{\bar n}_{{2}}\\
\tilde m_{41}^{(5)}&=&10\,{{\bar{\mu}}_{{1}}}^{4}{\bar{\mu}}_{{4}}{\bar n}_{{0}}+30\,{{\bar{\mu}}_{{1}}}^{4}{{\bar{\mu}}_{{2}}}^{2}{\bar n}_{{1}}+14\,{{\bar{\mu}}_{{1}}}^{6}{\bar{\mu}}_{{2}}{\bar n}_{{2}}+2\,{{\bar{\mu}}_{{1}}}^{8}{\bar n}_{{3}}+20\,{{\bar{\mu}}_{{1}}}^{2}{{\bar{\mu}}_{{2}}}^{3}{\bar n}_{{0}}
+12\,{{\bar{\mu}}_{{1}}}^{5}{\bar{\mu}}_{{3}}{\bar n}_{{1}}\\
&&+40\,{{\bar{\mu}}_{{1}}}^{3}{\bar{\mu}}_{{2}}{\bar{\mu}}_{{3}}{\bar n}_{{0}}\\
\tilde m_{41}^{(6)}&=&12\,{{\bar{\mu}}_{{1}}}^{5}{\bar{\mu}}_{{3}}{\bar n}_{{0}}+30\,{{\bar{\mu}}_{{1}}}^{4}{{\bar{\mu}}_{{2}}}^{2}{\bar n}_{{0}}+14\,{{\bar{\mu}}_{{1}}}^{6}{\bar{\mu}}_{{2}}{\bar n}_{{1}}+2\,{{\bar{\mu}}_{{1}}}^{8}{\bar n}_{{2}}\\
\tilde m_{41}^{(7)}&=&2\,{{\bar{\mu}}_{{1}}}^{8}{\bar n}_{{1}}
+14\,{{\bar{\mu}}_{{1}}}^{6}{\bar{\mu}}_{{2}}{\bar n}_{{0}}\\
\tilde m_{41}^{(8)}&=&2\,{{\bar{\mu}}_{{1}}}^{8}{\bar n}_{{0}}
\end{eqnarray*}
with
\begin{eqnarray*}
\alpha_{6,0}&=&2\,{\bar a}_{{0,1}}{\bar b}_{{6,0}}+2\,{\bar a}_{{1,1}}{\bar b}_{{5,0}}+2\,{\bar a}_{{2,1}}{\bar b}_{{4,0}}+2\,{\bar a}_{{3,1}
}{\bar b}_{{3,0}}+2\,{\bar a}_{{4,1}}{\bar b}_{{2,0}}+2\,{\bar a}_{{5,1}}{\bar b}_{{1,0}}+2\,{\bar a}_{{6,1}}{\bar b}_{{0,0}}\\
\alpha_{7,0}&=&2\,{\bar a}_{{0,1}}{\bar b}_{{7,0}}+2\,{\bar a}_{{1,1}}{\bar b}_{{6,0}}+2\,{\bar a}_{{2,1}}{\bar b}_{{5,0}}+2\,{\bar a}_{{3,1}
}{\bar b}_{{4,0}}+2\,{\bar a}_{{4,1}}{\bar b}_{{3,0}}+2\,{\bar a}_{{5,1}}{\bar b}_{{2,0}}+2\,{\bar a}_{{6,1}}{\bar b}_{{1,0}}\\
&&+2\,{\bar a}_{{7,1}}{\bar b}_{{0,0}}\\
\alpha_{8,0}&=&2\,{\bar a}_{{0,1}}{\bar b}_{{8,0}}+2\,{\bar a}_{{1,1}}{\bar b}_{{7,0}}+2\,{\bar a}_{{2,1}}{\bar b}_{{6,0}}+2\,{\bar a}_{{3,1}
}{\bar b}_{{5,0}}+2\,{\bar a}_{{4,1}}{\bar b}_{{4,0}}+2\,{\bar a}_{{5,1}}{\bar b}_{{3,0}}+2\,{\bar a}_{{6,1}}{\bar b}_{{2,0}}\\
&&+2\,{\bar a}_{{7,1}}{\bar b}_{{1,0}}+2\,{\bar a}_{{8,1}}{\bar b}_{{0,0}}\\
\alpha_{2,1}&=&2\,{{\bar a}_{{0,1}}}^{3}{\bar b}_{{2,2}}+6\,{{\bar a}_{{0,1}}}^{2}{\bar a}_{{1,1}}{\bar b}_{{1,2}}+6\,{{\bar a}_{{0,1}}}^{2}{
\bar a}_{{2,1}}{\bar b}_{{0,2}}+6\,{\bar a}_{{0,1}}{{\bar a}_{{1,1}}}^{2}{\bar b}_{{0,2}}+4\,{\bar a}_{{0,1}}{\bar a}_{{0,2}}{
\bar b}_{{2,1}}+2\,{\bar a}_{{2,3}}{\bar b}_{{0,0}}\\
&&+4\,{\bar a}_{{0,1}}{\bar a}_{{2,2}}{\bar b}_{{0,1}}+4\,{
\bar a}_{{0,2}}{\bar a}_{{1,1}}{\bar b}_{{1,1}}+4\,{\bar a}_{{0,2}}{\bar a}_{{2,1}}{\bar b}_{{0,1}}+4\,{\bar a}_{{1,1}}{\bar a}_{{1,2}}{\bar b}_{{0,1}}+2\,{\bar a}_{{0,3}}{\bar b}_{{2,0}}+2\,{\bar a}_{{1,3}}{\bar b}_{{1,0}}\\
&&+4\,{\bar a}_{{0,1}}{\bar a}_{{1,2}}{\bar b}_{{1,1}}\\
\gamma_{i,1}&=&\frac{1}{2}\alpha_{i,0},\ \ i=0,...,8,\\
\gamma_{i,3}&=&\frac{1}{2}\alpha_{i,1},\ \ i=0,1,2,
\end{eqnarray*}
\begin{eqnarray*}
\bar{\mu}_1&=&{\mu_{{1}}^{-1}}\\
\bar{\mu}_2&=&-{{\mu_{{1}}^{-3}}}{\mu_{{2}}}\\
\bar{\mu}_3&=&-{{\mu_{{1}}^{-5}}}({\mu_{{1}}\mu_{{3}}-2\,{\mu_{{2}}^{2}}})\\
\bar{\mu}_4&=&-{{\mu_{{1}}^{-7}}}({{\mu_{{1}}^{2}}\mu_{{4}}-5\,\mu_{{1}}\mu_{{2}}\mu_{{3}}+5\,{\mu_{{2}}^{3}}})\\
\bar{\mu}_5&=&-{{\mu_{{1}}^{-9}}}({{\mu_{{1}}^{3}}\mu_{{5}}-6\,{\mu_{{1}}^{2}}\mu_{{2}}\mu_{{4}}-3\,{\mu_{{1}}^{2}}{\mu_{{3}}^{2}}
+21\,\mu_{{1}}{\mu_{{2}}^{2}}\mu_{{3}}-14\,{\mu_{{2}}^{4}}})\\
\bar{\mu}_6&=&-{{\mu_{{1}}^{-11}}}({{\mu_{{1}}^{4}}\mu_{{6}}-7\,{\mu_{{1}}^{3}}\mu_{{2}}\mu_{{5}}-7\,{\mu_{{1}}^{3}}\mu_{{3}}\mu_{{4}}
+28\,{\mu_{{1}}^{2}}{\mu_{{2}}^{2}}\mu_{{4}}+28\,{\mu_{{1}}^{2}}\mu_{{2}}{\mu_{{3}}^{2}}+42\,{\mu_{{2}}}^{5}}\\
&&-84\,\mu_{{1}}{\mu_{{2}}^{3}}\mu_{{3}})\\
\bar{\mu}_7&=&-{{\mu_{{1}}^{-13}}}({{\mu_{{1}}}^{5}\mu_{{7}}-8\,{\mu_{{1}}^{4}}\mu_{{2}}\mu_{{6}}-8\,{\mu_{{1}}^{4}}\mu_{{3}}\mu_{{5}}
-4\,{\mu_{{1}}^{4}}{\mu_{{4}}^{2}}+36\,{\mu_{{1}}^{3}}{\mu_{{2}}^{2}}\mu_{{5}}+12\,{\mu_{{1}}^{3}}{\mu_{{3}}^{3}}}\\
&&{+72\,{\mu_{{1}}^{3}}\mu_{{2}}\mu_{{3}}\mu_{{4}}-120\,{\mu_{{1}}^{2}}{\mu_{{2}}^{3}}\mu_{{4}}-180\,{\mu_{{1}}^{2}}{\mu_{{2}}^{2}}{\mu_{{3}}^{2}}
+330\,\mu_{{1}}{\mu_{{2}}^{4}}\mu_{{3}}-132\,{\mu_{{2}}}^{6}})\\
\bar{\mu}_8&=&-{{\mu_{{1}}^{-15}}}({{\mu_{{1}}}^{6}\mu_{{8}}-9\,{\mu_{{1}}}^{5}\mu_{{2}}\mu_{{7}}-9\,{\mu_{{1}}}^{5}\mu_{{3}}\mu_{{6}}
-9\,{\mu_{{1}}}^{5}\mu_{{4}}\mu_{{5}}+45\,{\mu_{{1}}^{4}}{\mu_{{2}}^{2}}\mu_{{6}}+429\,{\mu_{{2}}^{7}}}\\
&&+90\,{\mu_{{1}}^{4}}\mu_{{2}}\mu_{{3}}\mu_{{5}}+45\,{\mu_{{1}}^{4}}\mu_{{2}}{\mu_{{4}}}^{2
}+45\,{\mu_{{1}}^{4}}{\mu_{{3}}^{2}}\mu_{{4}}-165\,{\mu_{{1}}^{3}}{\mu_{{2}}^{3}}\mu_{{5}}-495\,{\mu_{{1}}^{3}}{\mu_{{2}}^{2}}\mu_{{3}}\mu_{{4}}\\
&&-165\,{\mu_{{1}}^{3}}\mu_{{2}}{\mu_{{3}}^{3}}+495\,{\mu_{{1}}^{2}}{\mu_{{2}}^{4}}\mu_{{4}}+990\,{\mu_{{1}}^{2}}{\mu_{{2}}^{3}}{\mu_{{3}}
}^{2}-1287\,\mu_{{1}}{\mu_{{2}}}^{5}\mu_{{3}}),
\end{eqnarray*}
\begin{eqnarray*}
\bar n_0&=&{\mu_{{1}}^{-1}}\\
\bar n_1&=&-2\,{{\mu_{{1}}^{-2}}}{\mu_{{2}}}\\
\bar n_2&=&-{\,{{\mu_{{1}}^{-3}}}(3\mu_{{1}}\mu_{{3}}-4\,{\mu_{{2}}^{2}})}\\
\bar n_3&=&-4\,{{\mu_{{1}}^{-4}}}({{\mu_{{1}}^{2}}\mu_{{4}}-3\,\mu_{{1}}\mu_{{2}}\mu_{{3}}+2\,{\mu_{{2}}^{3}}})\\
\bar n_4&=&-{\,{{\mu_{{1}}^{-5}}}({5\mu_{{1}}^{3}}\mu_{{5}}-16\,{\mu_{{1}}^{2}}\mu_{{2}}\mu_{{4}}-9\,{\mu_{{1}}^{2}}{\mu_{{3}}^{2}}
+36\,\mu_{{1}}{\mu_{{2}}^{2}}\mu_{{3}}-16\,{\mu_{{2}}^{4}})}\\
\bar n_5&=&-2\,{{\mu_{{1}}^{-6}}}(3\,{\mu_{{1}}^{4}}\mu_{{6}}-10\,{\mu_{{1}}^{3}}\mu_{{2}}\mu_{{5}}-12\,{\mu_{{1}}^{3}}\mu_{{3}}\mu_{{4}}
+24\,{\mu_{{1}}^{2}}{\mu_{{2}}^{2}}\mu_{{4}}+27\,{\mu_{{1}}^{2}}\mu_{{2}}{\mu_{{3}}^{2}}\\
&&-48\,\mu_{{1}}{\mu_{{2}}^{3}}\mu_{{3}}+16\,{\mu_{{2}}^{5}})\\
\bar n_6&=&-\,{{\mu_{{1}}^{-7}}}({7\mu_{{1}}^{5}}\mu_{{7}}-24\,{\mu_{{1}}^{4}}\mu_{{2}}\mu_{{6}}-30\,{\mu_{{1}}^{4}}\mu_{{3}}\mu_{{5}}-16\,{\mu_{{1}}^{4}}{\mu_{{4
}}^{2}}+60\,{\mu_{{1}}^{3}}{\mu_{{2}}^{2}}\mu_{{5}}\\
&&+144\,{\mu_{{1}}^{3}}\mu_{{2}}\mu_{{3}}\mu_{{4}}+27\,{\mu_{{1}}^{3}}{\mu_{{3}}^{3}}-128\,
{\mu_{{1}}^{2}}{\mu_{{2}}^{3}}\mu_{{4}}-216\,{\mu_{{1}}^{2}}{\mu_{{2}}^{2}}{\mu_{{3}}^{2}}+240\,\mu_{{1}}{\mu_{{2}}^{4}}\mu_{{3}}-64\,{\mu_{{2}}^{6}})\\
\bar n_7&=&-4\,{{\mu_{{1}}^{-8}}}(2\,{\mu_{{1}}^{6}}\mu_{{8}}-7\,{\mu_{{1}}^{5}}\mu_{{2}}\mu_{{7}}-9\,{\mu_{{1}}^{5}}\mu_{{3}}\mu_{{6}}-10\,{\mu_{{1}}^{5}}\mu_{{4
}}\mu_{{5}}+18\,{\mu_{{1}}^{4}}{\mu_{{2}}^{2}}\mu_{{6}}\\
&&+45\,{\mu_{{1}}^{4}}\mu_{{2}}\mu_{{3}}\mu_{{5}}+24\,{\mu_{{1}}^{4}}\mu_{{2}}{\mu_{{4}
}}^{2}+27\,{\mu_{{1}}^{4}}{\mu_{{3}}^{2}}\mu_{{4}}-40\,{\mu_{{1}}^{3}}{\mu_{{2}}^{3}}\mu_{{5}}-144\,{\mu_{{1}}^{3}}{\mu_{{2}}^{2}}\mu_{{3}}
\mu_{{4}}\\
&&-54\,{\mu_{{1}}^{3}}\mu_{{2}}{\mu_{{3}}^{3}}+80\,{\mu_{{1}}^{2}}{\mu_{{2}}^{4}}\mu_{{4}}+180\,{\mu_{{1}}^{2}}{\mu_{{2}}^{3}}{\mu_{{
3}}^{2}}-144\,\mu_{{1}}{\mu_{{2}}^{5}}\mu_{{3}}+32\,{\mu_{{2}}^{7}})\\
\bar n_8&=&-\,{{\mu_{{1}}^{-9}}}(9{\mu_{{1}}^{7}}\mu_{{9}}-32\,{\mu_{{1}}^{6}}\mu_{{2}}\mu_{{8}}-42\,{\mu_{{1}}^{6}}\mu_{{3}}\mu_{{7}}
-48\,{\mu_{{1}}^{6}}\mu_{{4}}\mu_{{6}}-25\,{\mu_{{1}}^{6}}{\mu_{{5}}^{2}}\\
&&+84\,{\mu_{{1}}^{5}}{\mu_{{2}}^{2}}\mu_{{7}}+216\,{\mu_{{1}}^{5}}\mu_{{2}}\mu_{{3}}\mu_{{6}}+
240\,{\mu_{{1}}^{5}}\mu_{{2}}\mu_{{4}}\mu_{{5}}+135\,{\mu_{{1}}^{5}}{\mu_{{3}}^{2}}\mu_{{5}}+144\,{\mu_{{1}}^{5}}\mu_{{3}}{\mu_{{4}}^{2}}\\
&&-192\,{\mu_{{1}}^{4}}{\mu_{{2}}^{3}}\mu_{{6}}-720\,{\mu_{{1}}^{4}}{\mu_{{2}}^{2}}\mu_{{3}}\mu_{{5}}-384\,{\mu_{{1}}^{4}}{\mu_{{2}}^{2}}{\mu_{
{4}}^{2}}-864\,{\mu_{{1}}^{4}}\mu_{{2}}{\mu_{{3}}^{2}}\mu_{{4}}-81\,{\mu_{{1}}^{4}}{\mu_{{3}}^{4}}\\
&&+400\,{\mu_{{1}}^{3}}{\mu_{{2}}^{4}}\mu_{{5}}+1920\,{\mu_{{1}}^{3}}{\mu_{{2}}^{3}}\mu_{{3}}\mu_{{4}}+1080\,{\mu_{{1}}^{3}}{\mu_{{2}}^{2}}{\mu_{{3}}^{3}}
-768\,{\mu_{{1}}^{2}}{\mu_{{2}}^{5}}\mu_{{4}}-2160\,{\mu_{{1}}^{2}}{\mu_{{2}}^{4}}{\mu_{{3}}^{2}}\\
&&+1344\,\mu_{{1}}{\mu_{{2}}^{6}}\mu_{{3}}-256\,{\mu_{{2}}^{8}}),\\
\end{eqnarray*}
\begin{eqnarray*}
\mu_1&=&{(-h_{6})}^{\frac{1}{6}}\\
\mu_2&=&-\frac{1}{6}\,\left( -h_{{6}} \right) ^{-\frac{5}{6}}{h_{{7}}}\\
\mu_3&=&{\frac {1}{72}}\, \left( -h_{{6}} \right) ^{-{\frac {11}{6}}}\left( 12\,h_{{6}}h_{{8}}-5\,{h_{{7}}^{2}} \right)\\
\mu_4&=&-{\frac {1}{1296}}\, \left( -h_{{6}} \right) ^{-{\frac {17}{6}}}\left( 216\,{h_{{6}}^{2}}h_{{9}}-180\,h_{{6}}h_{{7}}h_{{8}}+55\,{h_{{7}}^{3}} \right)\\
\mu_5&=&{\frac {1}{31104}}\, \left( -h_{{6}} \right) ^{-{\frac {23}{6}}}\left( 5184\,{h_{{6}}^{3}}h_{{10}}-4320\,{h_{{6}}}^{2}h_{{7}}h_{{9}}-2160\,{h_{{6}}^{2}}{h_{{8}}^{2}}+3960\,h_{{6}}{h_{{7}}^{2}}h_{{8}}\right.\\
&&\left.-935\,{h_{{7}}^{4}} \right)\\
\mu_6&=&-{\frac {1}{186624}}\, \left( -h_{{6}} \right) ^{-{\frac {29}{6}}}\left( 31104\,{h_{{6}}^{4}}h_{{11}}-25920\,{h_{{6}}^{3}}h_{{7}}h_{{10}}-25920\,{h_{{6}}^{3}}h_{{8}}h_{{9}} +4301\,{h_{{7}}^{5}}\right.\\
&&\left.+23760\,{h_
{{6}}^{2}}{h_{{7}}^{2}}h_{{9}}+23760\,{h_{{6}}^{2}}h_{{7}}{h_{{8}}}^{2}-22440\,h_{{6}}{h_{{7}}^{3}}h_{{8}}\right)\\
\mu_7&=&{\frac {1}{6718464}}\, \left( -h_{{6}} \right) ^{-{\frac {35}{6}}}\left( 1119744\,{h_{{6}}^{5}}h_{{12}}-933120\,{h_{{6}}^{4}}h_{{7}}h_{{11}}-933120\,{h_{{6}}^{4}}h_{{8}}h_{{10}}\right.\\
&&+855360\,{h_{{6}}^{3}}{h_{{7}}^{2}}h_{{10}}+1710720\,{h_{{6}}^{3}}h_{{7}}h_{{8}}h_{{9}}+285120\,{h_{{6}}}^
{3}{h_{{8}}^{3}}-807840\,{h_{{6}}^{2}}{h_{{7}}^{3}}h_{{9}}\\
&&\left.-
466560\,{h_{{6}}^{4}}{h_{{9}}^{2}}-1211760\,{h_{{6}}^{2}}{h_{{7}}^{2}}{h_{{8}}^{2}}+774180\,h_{{6}}{h_{{7}}^{4}}h_{{8}}-124729\,{h_{{7}}^{6}} \right)\\
\mu_8&=&-{\frac {1}{40310784}}\, \left( -h_{{6}} \right) ^{-{\frac {41}{6}}}\left( 6718464\,{h_{{6}}^{6}}h_{{13}}-5598720\,{h_{{6}}^{5}}h_{{7}}h_{{12}}-5598720\,{h_{{6}}^{5}}h_{{8}}h_{{11}}\right.\\
&&-5598720\,{h_{{6}}^{5}}h_{{9}}h_{{10}}+5132160\,{h_{{6}}^{4}}{h_{{7}}}^{2}h_{{11}}+10264320\,{h_{{6}}^{4}}h_{{7}}h_{{8}}h_{{10}}+5132160\,{h_
{{6}}^{4}}h_{{7}}{h_{{9}}^{2}}\\
&&+5132160\,{h_{{6}}^{4}}{h_{{8}}^{2}}h_{{9}}-4847040\,{h_{{6}}^{3}}{h_{{7}}^{3}}h_{{10}}-14541120\,{h_{{6}}}^{3
}{h_{{7}}^{2}}h_{{8}}h_{{9}}-4847040\,{h_{{6}}^{3}}h_{{7}}{h_{{8}}}^{3}\\
&&\left.+4645080\,{h_{{6}}^{2}}{h_{{7}}^{4}}h_{{9}}+9290160\,{h_{{6}}^{2}}{h_
{{7}}^{3}}{h_{{8}}^{2}}-4490244\,h_{{6}}{h_{{7}}^{5}}h_{{8}} +623645\,{
h_{{7}}^{7}}\right)\\
\mu_9&=&{\frac {1}{1934917632}}\,\left( -h_{{6}} \right) ^{-{\frac {47}{6}}} \left( 322486272\,h_{{14}}{h_{{6}}^{7}}-
268738560\,h_{{7}}h_{{13}}{h_{{6}}^{6}}\right.\\
&&-268738560\,h_{{9}}h_{{11}}{h_{{6}}^{6}}-134369280\,{h_{{10}}
}^{2}{h_{{6}}^{6}}+246343680\,{h_{{7}}^{2}}h_{{12}}{h_{{6}}^{5}}\\
&&+492687360\,h_{{7}}h_{{8}}h_{{11}}{h_{{6}}^{5}}+492687360\,h_{{7}}h_{{9
}}h_{{10}}{h_{{6}}^{5}}+246343680\,{h_{{8}}^{2}}h_{{10}}{h_{{6}}^{5}}\\
&&+246343680\,h_{{8}}{h_{{9}}^{2}}{h_{{6}}^{5}}-232657920\,{h_{{7}}^{3}}h
_{{11}}{h_{{6}}^{4}}-697973760\,{h_{{7}}^{2}}h_{{8}}h_{{10}}{h_{{6}}}^{4}\\
&&-348986880\,{h_{{7}}^{2}}{h_{{9}}^{2}}{h_{{6}}^{4}}-697973760\,h_{{
7}}{h_{{8}}^{2}}h_{{9}}{h_{{6}}^{4}}-58164480\,{h_{{8}}^{4}}{h_{{6}}}^{4}\\
&&+222963840\,{h_{{7}}^{4}}h_{{10}}{h_{{6}}^{3}}+891855360\,{h_{{7}}}
^{3}h_{{8}}h_{{9}}{h_{{6}}^{3}}+445927680\,{h_{{7}}^{2}}{h_{{8}}^{3}}{h_{{6}}^{3}}\\
&&-215531712\,{h_{{7}}^{5}}h_{{9}}{h_{{6}}^{2}}-538829280\,{
h_{{7}}^{4}}{h_{{8}}^{2}}{h_{{6}}^{2}}+209544720\,{h_{{7}}^{6}}h_{{8}}h_{{6}}\\
&&\left.-25569445\,{h_{{7}}^{8}}-268738560\,h_{{8}}h_{{12}}{h_{{6}}^{6}} \right),
\end{eqnarray*}
and

\begin{eqnarray*}
\bar a_{6,1}&=&\frac{\sqrt {2}}{16}\left( 3240\,{h^{3}_{{0,3}}}{h^{3}_{{2,1}}}+17820\,{h^{2}_{{0,3}}}{h^{2}_{{1,2}}}{h^{2}_{{2,1}}}+231\,{h^{6}_{{1,2}}}+6930\,h_{{0,3}}{h^{4}_{{1,2}}}h_{{2,1}}\right.\\
&&+48\,h_{{0,3}}h_{{6,1}}-6480\,{h^{2}_{{0,3}}}h_{{1,2}}h_{{2,1}}h_{{3,1}}-3240\,{h^{2}_{{0,3}}}{h^{2}_{{2,1}}}h_{{2,2}}-1632\,h_{{0,3}}h_{{0,4}}{h^{3}_{{2,1}}}\\
&&-40\,{h^{3}_{{2,2}}}-2520\,h_{{0,3}}{h^{3}_{{1,2}}}h_{{3,1}}-7560\,h_{{0,3}}{h^{2}_{{1,2}}}h_{{2,1}}h_{{2,2}}-6480\,h_{{0,3}}h_{{1,2}}h_{{1,3}}{h^{2}_{{2,1}}}\\
&&-3024\,h_{{0,4}}{h^{2}_{{1,2}}}{h^{2}_{{2,1}}}-630\,{h^{4}_{{1,2}}}h_{{2,2}}-2520\,{h^{3}_{{1,2}}}h_{{1,3}}h_{{2,1}}+720\,{h^{2}_{{0,3}}}h_{{2,1}}h_{{4,1}}\\
&&+48\,h_{{1,2}}h_{{5,2}}+840\,h_{{0,3}}{h^{2}_{{1,2}}}h_{{4,1}}+1680\,h_{{0,3}}h_{{1,2
}}h_{{2,1}}h_{{3,2}}+1680\,h_{{0,3}}h_{{1,2}}h_{{2,2}}h_{{3,1}}\\
&&+1440\,
h_{{0,3}}h_{{1,3}}h_{{2,1}}h_{{3,1}}+720\,h_{{0,3}}{h^{2}_{{2,1}}}h_{{
2,3}}+840\,h_{{0,3}}h_{{2,1}}{h^{2}_{{2,2}}}+840\,{h^{2}_{{1,2}}}h_{
{1,3}}h_{{3,1}}\\
&&+672\,h_{{0,4}}{h^{2}_{{2,1}}}h_{{2,2}}+160\,h_{{0,5}
}{h^{3}_{{2,1}}}+1344\,h_{{0,4}}h_{{1,2}}h
_{{2,1}}h_{{3,1}}+280\,{h^{3}_{{1,2}}}h_{{3,2}}\\
&&+840\,{h^{2}_{{1,2}}}h_{{2,1}}h_{{2,3}}+420\,{h^{2}_{{1,2}}
}{h^{2}_{{2,2}}}+1680\,h_{{1,2}}h_{{1,3}}h_{{2,1}}h_{{2,2}}+672\,h
_{{1,2}}h_{{1,4}}{h^{2}_{{2,1}}}\\
&&+360\,{h^{2}_{{1,3}}}{h^{2}_{{2,1}}}-
240\,h_{{0,3}}h_{{1,2}}h_{{5,1}}-240\,h_{{0,3}}h_{{2,1}}h_{{4,2}}-240
\,h_{{0,3}}h_{{2,2}}h_{{4,1}}\\
&&-192\,h_
{{0,4}}h_{{2,1}}h_{{4,1}}-240\,h_{{1,2}}h_{{1,3}}h_{{4,1}}-240\,h_{{1,2}}h_{{2,1}
}h_{{3,3}}-240\,h_{{1,2}}h_{{2,2}}h_{{3,2}}\\
&&-240\,h_{{1,2}}h_{{2,3}}h_{
{3,1}}-240\,h_{{1,3}}h_{{2,1}}h_{{3,2}}-240\,h_{{1,3}}h_{{2,2}}h_{{3,1
}}-240\,
h_{{2,1}}h_{{2,2}}h_{{2,3}}\\
&&
+48\,h_{{1,3}}h_{{5,1}}+48\,h_{{2,1}}h_{{4,3}}+
48\,h_{{2,2}}h_{{4,2}}+
24\,{h^{2}_{{3,2}}}-16\,h_{{6,2}}+360\,{h^{2}_{{0,3}}}{h^{2}_{{3,1}}}\\
&&+48\,h_{{2,3}}h_{{4,1}}-192\,h_{{1,4}}h_{{2,1}}h_{{3,1}}-96\,h_{{0,4}}{h^{2}_{{3,1}}}-96\,{h^{2}_{{2,1}}}h_{{2,4}}-240\,h_{{0,3}}h_{{3,1}}h_{{3,2}}\\
&&\left.+48\,h_{{3,1}}h_{{3,3}}-120\,{h^{2}_{{1,2}}
}h_{{4,2}}\right)\\
\bar a_{7,1}&=&-\frac{1}{16}\,\sqrt {2} \left( -160\,h_{{1,5}}{h_{{2,1}}^{3}}-48\,h_{{5,3}}h_{{2,1}}+16\,h_{{7,2}}+96\,h_{{3,4}}{h_{{2,1}}^{2}}-1386\,{h_{{1,2}}^{5}}h_{{2,2}}\right.\\
&&-48\,h_{{1,3}}h_{{6,1}}-48\,h_{{2,3}}h_{{5,1}}+96\,h_{{1,4}}{h_{{3,1}}^{2}}-48\,h_{{3,3}}h_{{4,1}}-48\,h_{{3,1}}h_{{4,3}}-48\,h_{{0,3}}h_{{7,1}}\\
&&+120\,{h_{{1,2}}^{2}}h_{{5,2}}+120\,h_{{1,2}}{h_{{3,2}}^{2}}+120\,{h_{{2,2}}^{2}}h_{{3,2}}+630\,{h_{{1,2}}^{4}}h_{{3,2}}+1260\,{h_{{1,2}}^{3}}{h_{{2,2}}^{2}}\\
&&-280\,{h_{{1,2}}^{3}}h_{{4,2}}-48\,h_{{3,2}}h_{{4,2}}-48\,h_{{2,2}}h_{{5,2}}-48\,h_{{1,2}}h_{{6,2}}+429\,{h_{{1,2}}^{7}}-280\,h_{{1,2}}{h_{{2,2}}^{3}}\\
&&+77220\,{h_{{0,3}}^{2}}{h_{{1,2}}^{3}}{h_{{2,1}}^{2}}+18018\,h_{{0,3}}{h_{{1,2}}^{5}}h_{{2,1}}-9720\,{h_{{0,3}}^{3}}{h_{{2,1}}^{2}}h_{{3,1}}-9720\,{h_{{0,3}}^{2}}h_{{1,3}}{h_{{2,1}}^{3}}\\
&&\left.-6930\,h_{{0,3}}{h_{{1,2}}^{4}}h_{{3,1}}+3240\,{h_{{0,3}}^{2}}h_{{1,2}}{h_{{3,1}}^{2}}+3240\,{h_{{0,3}}^{2}}{h_{{2,1}}^{2}}h_{{3,2}}+2520\,h_{{0,3}}{h_{{1,2}}^{3}}h_{{4,1}}\right.\\
&&+1632\,h_{{0,3}}h_{{1,4}}{h_{{2,1}}^{3}}-720\,{h_{{0,3}}^{2}}h_{{2,1}}h_{{5,1}}-720\,{h_{{0,3}}^{2}}h_{{3,1}}h_{{4,1}}-840\,h_{{0,3}}{h_{{1,2}}^{2}}h_{{5,1}}\\
&&-720\,h_{{0,3}}h_{{1,3}}{h_{{3,1}}^{2}}-720\,h_{{0,3}}{h_{{2,1}}^{2}}h_{{3,3}}-840\,h_{{0,3}}{h_{{2,2}}^{2}}h_{{3,1}}+240\,h_{{0,3}}h_{{1,2}}h_{{6,1}}\\
&&+240\,h_{{0,3}}h_{{2,1}}h_{{5,2}}+240\,h_{{0,3}}h_{{2,2}}h_{{5,1}}+240\,h_{{0,3}}h_{{3,1}}h_{{4,2}}+240\,h_{{0,3}}h_{{3,2}}h_{{4,1}}\\
&&-11088\,h_{{0,4}}{h_{{1,2}}^{3}}{h_{{2,1}}^{2}}+1632\,h_{{0,4}}h_{{1,3}}{h_{{2,1}}^{3}}-672\,h_{{0,4}}h_{{1,2}}{h_{{3,1}}^{2}}-672\,h_{{0,4}}{h_{{2,1}}^{2}}h_{{3,2}}\\
&&+192\,h_{{0,4}}h_{{2,1}}h_{{5,1}}+192\,h_{{0,4}}h_{{3,1}}h_{{4,1}}+1440\,h_{{0,5}}h_{{1,2}}{h_{{2,1}}^{3}}-480\,h_{{0,5}}{h_{{2,1}}}^{2}h_{{3,1}}\\
&&-6930\,{h_{{1,2}}^{4}}h_{{1,3}}h_{{2,1}}+2520\,{h_{{1,2}}^{3}}h_{{1,3}}h_{{3,1}}+3240\,h_{{1,2}}{h_{{1,3}}^{2}}{h_{{2,1}}^{2}}-840\,{h_{{1,2}}^{2}}h_{{1,3}}h_{{4,1}}\\
&&-720\,{h_{{1,3}}^{2}}h_{{2,1}}h_{{3,1}}-720\,h_{{1,3}}{h_{{2,1}}^{2}}h_{{2,3}}-840\,h_{{1,3}}h_{{2,1}}{h_{{2,2}}^{2}}+240\,h_{{1,2}}h_{{1,3}}h_{{5,1}}\\
&&+240\,h_{{1,3}}h_{{2,1}}h_{{4,2}}+240\,h_{{1,3}}h_{{2,2}}h_{{4,1}}+240\,h_{{1,3}}h_{{3,1}}h_{{3,2}}+2520\,{h_{{1,2}}^{3}}h_{{2,1}}h_{{2,3}}\\
&&-840\,{h_{{1,2}}^{2}}h_{{2,3}}h_{{3,1}}+240\,h_{{1,2}}h_{{2,3}}h_{{4,1}}+240\,h_{{2,1}}h_{{2,3}}h_{{3,2}}+240\,h_{{2,2}}h_{{2,3}}h_{{3,1}}\\
&&+3024\,{h_{{1,2}}^{2}}h_{{1,4}}{h_{{2,1}}^{2}}-672\,h_{{1,4}}{h_{{2,1}}^{2}}h_{{2,2}}+192\,h_{{1,4}}h_{{2,1}}h_{{4,1}}-840\,{h_{{1,2}}^{2}}h_{{2,1}}h_{{3,3}}\\
&&+240\,h_{{1,2}}h_{{3,1}}h_{{3,3}}+240\,h_{{2,1}}h_{{2,2}}h_{{3,3}}-672\,h_{{1,2}}{h_{{2,1}}^{2}}h_{{2,4}}+192\,h_{{2,1}}h_{{2,4}}h_{{3,1}}\\
&&+240\,h_{{1,2}}h_{{2,1}}h_{{4,3}}+240\,h_{{1,2}}h_{{2,2}}h_{{4,2}}-840\,{h_{{1,2}}^{2}}h_{{2,2}}h_{{3,2}}-35640\,{h_{{0,3}}^{2}}{h_{{1,2}}^{2}}h_{{2,1}}h_{{3,1}}\\
&&-35640\,{h_{{0,3}}^{2}}h_{{1,2}}{h_{{2,1}}^{2}}h_{{2,2}}-17952\,h_{{0,3}}h_{{0,4}}h_{{1,2}}{h_{{2,1}}^{3}}-27720\,h_{{0,3}}{h_{{1,2}}^{3}}h_{{2,1}}h_{{2,2}}\\
&&-35640\,h_{{0,3}}{h_{{1,2}}^{2}}h_{{1,3}}{h_{{2,1}}^{2}}+6480\,{h_{{0,3}}^{2}}h_{{1,2}}h_{{2,1}}h_{{4,1}}+6480\,{h_{{0,3}}^{2}}h_{{2,1}}h_{{2,2}}h_{{3,1}}\\
&&+4896\,h_{{0,3}}h_{{0,4}}{h_{{2,1}}^{2}}h_{{3,1}}+7560\,h_{{0,3}}{h_{{1,2}}^{2}}h_{{2,1}}h_{{3,2}}+7560\,h_{{0,3}}{h_{{1,2}}^{2}}h_{{2,2}}h_{{3,1}}\\
&&+6480\,h_{{0,3}}h_{{1,2}}{h_{{2,1}}^{2}}h_{{2,3}}+7560\,h_{{0,3}}h_{{1,2}}h_{{2,1}}{h_{{2,2}}^{2}}+6480\,h_{{0,3}}h_{{1,3}}{h_{{2,1}}^{2}}h_{{2,2}}\\
&&-1680\,h_{{0,3}}h_{{1,2}}h_{{2,1}}h_{{4,2}}-1680\,h_{{0,3}}h_{{1,2}}h_{{2,2}}h_{{4,1}}-1680\,h_{{0,3}}h_{{1,2}}h_{{3,1}}h_{{3,2}}\\
&&-1440\,h_{{0,3}}h_{{1,3}}h_{{2,1}}h_{{4,1}}-1680\,h_{{0,3}}h_{{2,1}}h_{{2,2}}h_{{3,2}}-1440\,h_{{0,3}}h_{{2,1}}h_{{2,3}}h_{{3,1}}\\
&&+6048\,h_{{0,4}}{h_{{1,2}}^{2}}h_{{2,1}}h_{{3,1}}+6048\,h_{{0,4}}h_{{1,2}}{h_{{2,1}}^{2}}h_{{2,2}}-1344\,h_{{0,4}}h_{{1,2}}h_{{2,1}}h_{{4,1}}\\
&&-1344\,h_{{0,4}}h_{{2,1}}h_{{2,2}}h_{{3,1}}+7560\,{h_{{1,2}}^{2}}h_{{1,3}}h_{{2,1}}h_{{2,2}}-1680\,h_{{1,2}}h_{{1,3}}h_{{2,1}}h_{{3,2}}\\
&&-1680\,h_{{1,2}}h_{{1,3}}h_{{2,2}}h_{{3,1}}-1680\,h_{{1,2}}h_{{2,1}}h_{{2,2}}h_{{2,3}}-1344\,h_{{1,2}}h_{{1,4}}h_{{2,1}}h_{{3,1}}\\
&&\left.+12960\,h_{{0,3}}h_{{1,2}}h_{{1,3}}h_{{2,1}}h_{{3,1}}+42120\,{h_{{0,3}}^{3}}h_{{1,2}}{h_{{2,1}}^{3}} \right)\\
\bar a_{8,1}&=&{\frac {1}{128}}\,\sqrt {2} \left( 1280\,h_{{2,5}}{h_{{2,1}}^{3}}+384\,h_{{6,3}}h_{{2,1}}-128\,h_{{8,2}}-1920\,h_{{0,6}}{h_{{2,1}}^{4}}-768
\,h_{{4,4}}{h_{{2,1}}^{2}}\right.\\
&&+384\,h_{{2,2}}h_{{6,2}}-24024\,{h_{{1,2}}}^{6}h_{{2,2}}+11088\,{h_{{1,2}}^{5}}h_{{3,2}}+27720\,{h_{{1,2}}^{4}}{h_{{2,2}}^{2}}-960\,{h_{{2,2}}^{2}}h_{{4,2}}\\
&&-960\,h_{{2,2}}{h_{{3,2}}^{2}}-10080\,{h_{{1,2}}^{2}}{h_{{2,2}}^{3}}-5040\,{h_{{1,2}}^{4}}h_{{4,2}}+3360\,{h_{{1,2}}^{2}}{h_{{3,2}}^{2}}+2240\,{h_{{1,2}}^{3}}h_{{5,2}}\\
&&-25920\,{h_{{0,3}}^{2}}h_{{2,2}}{h_{{3,1}}^{2}}+13056\,{h_{{0,4}}^{2}}{h_{{2,1}}^{4}}-768\,h_{{0,4}}{h_{{4,1}}^{2}}+2880\,{h_{{1,3}}^{2}}{h_{{3,1}}^{2}}+384\,h_{{1,3}}h_{{7,1}}\\
&&+2880\,{h_{{2,1}}^{2}}{h_{{2,3}}^{2}}+384\,h_{{2,3}}h_{{6,1}}+384\,h_{{3,3}}h_{{5,1}}-768\,h_{{2,4}}{h_{{3,1}}^{2}}+384\,h_{{4,1}}h_{{4,3}}\\
&&+560\,{h_{{2,2}}^{4}}+192\,{h_{{4,2}}^{2}}+6435\,{h_{{1,2}}^{8}}-20160\,{h_{{1,2}}^{3}}h_{{2,2}}h_{{3,2}}+6720\,{h_{{1,2}}^{2}}h_{{2,2}}h_{{4,2}}\\
&&+6720\,h_{{1,2}}{h_{{2,2}}^{2}}h_{{3,2}}-1920\,h_{{1,2}}h_{{2,2}}h_{{5,2}}-1920\,h_{{1,2}}h_{{3,2}}h_{{4,2}}-20160\,{h_{{1,2}}^{3}}h_{{2,1}}h_{{3,3}}\\
&&+6720\,{h_{{1,2}}^{2}}h_{{3,1}}h_{{3,3}}-1920\,h_{{1,2}}h_{{3,3}}h_{{4,1}}-1920\,h_{{2,1}}h_{{3,2}}h_{{3,3}}-1920\,h_{{2,2}}h_{{3,1}}h_{{3,3}}\\
&&-24192\,{h_{{1,2}}^{2}}{h_{{2,1}}^{2}}h_{{2,4}}+5376\,{h_{{2,1}}^{2}}h_{{2,2}}h_{{2,4}}-1536\,h_{{2,1}}h_{{2,4}}h_{{4,1}}-11520\,h_{{1,2}}h_{{1,5}}{h_{{2,1}}^{3}}\\
&&+384\,h_{{3,1}}h_{{5,3}}+3840\,h_{{1,5}}{h_{{2,1}}^{2}}h_{{3,1}}+6720\,{h_{{1,2}}^{2}}h_{{2,1}}h_{{4,3}}-1920\,h_{{1,2}}h_{{3,1}}h_{{4,3}}\\
&&-1920\,h_{{2,1}}h_{{2,2}}h_{{4,3}}+5376\,h_{{1,2}}{h_{{2,1}}^{2}}h_{{3,4}}-1536\,h_{{2,1}}h_{{3,1}}h_{{3,4}}-1920\,h_{{1,2}}h_{{2,1}}h_{{5,3}}\\
&&-20160\,h_{{0,3}}h_{{2,1}}{h_{{2,2}}^{3}}+6720\,h_{{0,3}}{h_{{2,2}}^{2}}h_{{4,1}}-1920\,h_{{0,3}}h_{{2,2}}h_{{6,1}}-25920\,{h_{{0,3}}^{2}}{h_{{2,1}}^{2}}h_{{4,2}}\\
&&-336960\,{h_{{0,3}}^{3}}{h_{{2,1}}^{3}}h_{{2,2}}+5760\,{h_{{0,3}}}^{2}h_{{2,1}}h_{{6,1}}+142560\,{h_{{0,3}}^{2}}{h_{{2,1}}^{2}}{h_{{2,2}}^{2}}+384\,h_{{0,3}}h_{{8,1}}\\
&&+77760\,{h_{{0,3}}^{3}}{h_{{2,1}}^{2}}h_{{4,1}}+77760\,{h_{{0,3}}^{2}}{h_{{2,1}}^{3}}h_{{2,3}}-185472\,{h_{{0,3}}^{2}}h_{{0,4}}{h_{{2,1}}^{4}}+384\,h_{{3,2}}h_{{5,2}}\\
&&+77760\,{h_{{0,3}}^{3}}h_{{2,1}}{h_{{3,1}}^{2}}+5760\,{h_{{0,3}}^{2}}h_{{3,1}}h_{{5,1}}+2527200\,{h_{{0,3}}^{3}}{h_{{1,2}}^{2}}{h_{{2,1}}^{3}}-960\,{h_{{1,2}}^{2}}h_{{6,2}}\\
&&+2316600\,{h_{{0,3}}^{2}}{h_{{1,2}}^{4}}{h_{{2,1}}^{2}}+142560\,{h_{{0,3}}^{2}}{h_{{1,2}}^{2}}{h_{{3,1}}^{2}}-288288\,h_{{0,4}}{h_{{1,2}}^{4}}{h_{{2,1}}^{2}}+384\,h_{{1,2}}h_{{7,2}}\\
&&-24192\,h_{{0,4}}{h_{{1,2}}^{2}}{h_{{3,1}}^{2}}-13056\,h_{{0,4}}{h_{{2,1}}^{3}}h_{{2,3}}-24192\,h_{{0,4}}{h_{{2,1}}^{2}}{h_{{2,2}}^{2}}+252720\,{h_{{0,3}}^{4}}{h_{{2,1}}^{4}}\\
&&+5376\,h_{{0,4}}{h_{{2,1}}^{2}}h_{{4,2}}+5376\,h_{{0,4}}h_{{2,2}}{h_{{3,1}}^{2}}-1536\,h_{{0,4}}h_{{2,1}}h_{{6,1}}-1536\,h_{{0,4}}h_{{3,1}}h_{{5,1}}\\
&&+63360\,h_{{0,5}}{h_{{1,2}}^{2}}{h_{{2,1}}^{3}}-11520\,h_{{0,5}}{h_{{2,1}}^{3}}h_{{2,2}}+3840\,h_{{0,5}}{h_{{2,1}}^{2}}h_{{4,1}}+3840\,h_{{0,5}}h_{{2,1}}{h_{{3,1}}^{2}}\\
&&-144144\,{h_{{1,2}}^{5}}h_{{1,3}}h_{{2,1}}+55440\,{h_{{1,2}}^{4}}h_{{1,3}}h_{{3,1}}+142560\,{h_{{1,2}}^{2}}{h_{{1,3}}^{2}}{h_{{2,1}}^{2}}+2880\,{h_{{0,3}}^{2}}{h_{{4,1}}^{2}}\\
&&-20160\,{h_{{1,2}}^{3}}h_{{1,3}}h_{{4,1}}-25920\,{h_{{1,3}}^{2}}{h_{{2,1}}^{2}}h_{{2,2}}-13056\,h_{{1,3}}h_{{1,4}}{h_{{2,1}}^{3}}+6720\,{h_{{1,2}}^{2}}h_{{1,3}}h_{{5,1}}\\
&&+5760\,{h_{{1,3}}^{2}}h_{{2,1}}h_{{4,1}}+5760\,h_{{1,3}}{h_{{2,1}}^{2}}h_{{3,3}}+6720\,h_{{1,3}}{h_{{2,2}}^{2}}h_{{3,1}}-1920\,h_{{1,2}}h_{{1,3}}h_{{6,1}}\\
&&-1920\,h_{{1,3}}h_{{2,1}}h_{{5,2}}-1920\,h_{{1,3}}h_{{2,2}}h_{{5,1}}-1920\,h_{{1,3}}h_{{3,1}}h_{{4,2}}-1920\,h_{{1,3}}h_{{3,2}}h_{{4,1}}\\
&&+55440\,{h_{{1,2}}^{4}}h_{{2,1}}h_{{2,3}}-20160\,{h_{{1,2}}^{3}}h_{{2,3}}h_{{3,1}}+6720\,{h_{{1,2}}^{2}}h_{{2,3}}h_{{4,1}}+6720\,h_{{2,1}}{h_{{2,2}}^{2}}h_{{2,3}}\\
&&-1920\,h_{{1,2}}h_{{2,3}}h_{{5,1}}-1920\,h_{{2,1}}h_{{2,3}}h_{{4,2}}-1920\,h_{{2,2}}h_{{2,3}}h_{{4,1}}-1920\,h_{{2,3}}h_{{3,1}}h_{{3,2}}\\
&&+88704\,{h_{{1,2}}^{3}}h_{{1,4}}{h_{{2,1}}^{2}}+5376\,h_{{1,2}}h_{{1,4}}{h_{{3,1}}^{2}}+5376\,h_{{1,4}}{h_{{2,1}}^{2}}h_{{3,2}}-1536\,h_{{1,4}}h_{{2,1}}h_{{5,1}}\\
&&-1536\,h_{{1,4}}h_{{3,1}}h_{{4,1}}-13056\,h_{{0,3}}h_{{2,4}}{h_{{2,1}}^{3}}+5760\,h_{{0,3}}h_{{4,3}}{h_{{2,1}}^{2}}+24960\,h_{{0,3}}h_{{0,5}}{h_{{2,1}}^{4}}\\
&&-1920\,h_{{0,3}}h_{{6,2}}h_{{2,1}}+77760\,h_{{0,3}}{h_{{1,3}}}^{2}{h_{{2,1}}^{3}}+5760\,h_{{0,3}}h_{{2,3}}{h_{{3,1}}^{2}}+6720\,h_{{0,3}}h_{{2,1}}{h_{{3,2}}^{2}}\\
&&-1920\,h_{{0,3}}h_{{3,2}}h_{{5,1}}-1920\,h_{{0,3}}h_{{3,1}}h_{{5,2}}+360360\,h_{{0,3}}{h_{{1,2}}^{6}}h_{{2,1}}-144144\,h_{{0,3}}{h_{{1,2}}^{5}}h_{{3,1}}\\
&&+55440\,h_{{0,3}}{h_{{1,2}}}^{4}h_{{4,1}}-20160\,h_{{0,3}}{h_{{1,2}}^{3}}h_{{5,1}}+6720\,h_{{0,3}}{h_{{1,2}}^{2}}h_{{6,1}}-1920\,h_{{0,3}}h_{{1,2}}h_{{7,1}}\\
&&-1920\,h_{{0,3}}h_{{4,1}}h_{{4,2}}-60480\,{h_{{1,2}}^{2}}h_{{1,3}}h_{{2,1}}h_{{3,2}}-60480\,{h_{{1,2}}^{2}}h_{{1,3}}h_{{2,2}}h_{{3,1}}\\
&&-60480\,h_{{1,2}}h_{{1,3}}h_{{2,1}}{h_{{2,2}}^{2}}+13440\,h_{{1,2}}h_{{1,3}}h_{{2,1}}h_{{4,2}}+13440\,h_{{1,2}}h_{{1,3}}h_{{2,2}}h_{{4,1}}\\
&&+13440\,h_{{1,2}}h_{{1,3}}h_{{3,1}}h_{{3,2}}+13440\,h_{{1,3}}h_{{2,1}}h_{{2,2}}h_{{3,2}}+11520\,h_{{1,3}}h_{{2,1}}h_{{2,3}}h_{{3,1}}\\
&&-51840\,h_{{1,2}}{h_{{1,3}}^{2}}h_{{2,1}}h_{{3,1}}-51840\,h_{{1,2}}h_{{1,3}}{h_{{2,1}}}^{2}h_{{2,3}}-60480\,{h_{{1,2}}^{2}}h_{{2,1}}h_{{2,2}}h_{{2,3}}\\
&&+13440\,h_{{1,2}}h_{{2,1}}h_{{2,3}}h_{{3,2}}+13440\,h_{{1,2}}h_{{2,2}}h_{{2,3}}h_{{3,1}}-48384\,{h_{{1,2}}^{2}}h_{{1,4}}h_{{2,1}}h_{{3,1}}\\
&&+10752\,h_{{1,2}}h_{{1,4}}h_{{2,1}}h_{{4,1}}+10752\,h_{{1,4}}h_{{2,1}}h_{{2,2}}h_{{3,1}}+13440\,h_{{1,2}}h_{{2,1}}h_{{2,2}}h_{{3,3}}\\
&&-48384\,h_{{1,2}}h_{{1,4}}{h_{{2,1}}^{2}}h_{{2,2}}+10752\,h_{{1,2}}h_{{2,1}}h_{{2,4}}h_{{3,1}}-51840\,{h_{{0,3}}^{2}}h_{{1,2}}h_{{3,1}}h_{{4,1}}\\
&&-1010880\,{h_{{0,3}}^{2}}h_{{1,2}}h_{{1,3}}{h_{{2,1}}^{3}}+285120\,{h_{{0,3}}^{2}}{h_{{1,2}}^{2}}h_{{2,1}}h_{{4,1}}-51840\,{h_{{0,3}}^{2}}h_{{1,2}}h_{{2,1}}h_{{5,1}}\\
&&-51840\,{h_{{0,3}}^{2}}h_{{2,1}}h_{{2,2}}h_{{4,1}}+233280\,{h_{{0,3}}^{2}}h_{{1,3}}{h_{{2,1}}^{2}}h_{{3,1}}+285120\,{h_{{0,3}}^{2}}h_{{1,2}}{h_{{2,1}}^{2}}h_{{3,2}}\\
&&-51840\,{h_{{0,3}}^{2}}h_{{2,1}}h_{{3,1}}h_{{3,2}}-1010880\,{h_{{0,3}}^{3}}h_{{1,2}}{h_{{2,1}}^{2}}h_{{3,1}}-1235520\,{h_{{0,3}}^{2}}{h_{{1,2}}^{3}}h_{{2,1}}h_{{3,1}}\\
&&-1853280\,{h_{{0,3}}^{2}}{h_{{1,2}}^{2}}{h_{{2,1}}^{2}}h_{{2,2}}-60480\,h_{{0,3}}{h_{{1,2}}^{2}}h_{{2,2}}h_{{4,1}}-51840\,h_{{0,3}}h_{{1,2}}h_{{1,3}}{h_{{3,1}}^{2}}\\
&&-51840\,h_{{0,3}}h_{{1,2}}{h_{{2,1}}^{2}}h_{{3,3}}-60480\,h_{{0,3}}h_{{1,2}}{h_{{2,2}}^{2}}h_{{3,1}}+13440\,h_{{0,3}}h_{{1,2}}h_{{2,2}}h_{{5,1}}\\
&&+13440\,h_{{0,3}}h_{{1,2}}h_{{3,1}}h_{{4,2}}+13440\,h_{{0,3}}h_{{2,1}}h_{{2,2}}h_{{4,2}}+143616\,h_{{0,3}}h_{{0,4}}{h_{{2,1}}}^{3}h_{{2,2}}\\
&&+11520\,h_{{0,3}}h_{{1,3}}h_{{3,1}}h_{{4,1}}-51840\,h_{{0,3}}{h_{{2,1}}^{2}}h_{{2,2}}h_{{2,3}}+11520\,h_{{0,3}}h_{{2,1}}h_{{3,1}}h_{{3,3}}\\
&&-39168\,h_{{0,3}}h_{{1,4}}{h_{{2,1}}^{2}}h_{{3,1}}+11520\,h_{{0,3}}h_{{1,3}}h_{{2,1}}h_{{5,1}}-39168\,h_{{0,3}}h_{{0,4}}h_{{2,1}}{h_{{3,1}}^{2}}\\
&&+11520\,h_{{0,3}}h_{{2,1}}h_{{2,3}}h_{{4,1}}-39168\,h_{{0,3}}h_{{0,4}}{h_{{2,1}}^{2}}h_{{4,1}}+285120\,h_{{0,3}}{h_{{1,2}}}^{2}{h_{{2,1}}^{2}}h_{{2,3}}\\
&&+221760\,h_{{0,3}}{h_{{1,2}}^{3}}h_{{2,1}}h_{{3,2}}-60480\,h_{{0,3}}{h_{{1,2}}^{2}}h_{{3,1}}h_{{3,2}}-51840\,h_{{0,3}}h_{{1,3}}{h_{{2,1}}^{2}}h_{{3,2}}\\
&&+13440\,h_{{0,3}}h_{{1,2}}h_{{3,2}}h_{{4,1}}+13440\,h_{{0,3}}h_{{2,2}}h_{{3,1}}h_{{3,2}}+13440\,h_{{0,3}}h_{{1,2}}h_{{2,1}}h_{{5,2}}\\
&&-933504\,h_{{0,3}}h_{{0,4}}{h_{{1,2}}}^{2}{h_{{2,1}}^{3}}-720720\,h_{{0,3}}{h_{{1,2}}^{4}}h_{{2,1}}h_{{2,2}}-1235520\,h_{{0,3}}{h_{{1,2}}^{3}}h_{{1,3}}{h_{{2,1}}^{2}}\\
&&+221760\,h_{{0,3}}{h_{{1,2}}^{3}}h_{{2,2}}h_{{3,1}}+332640\,h_{{0,3}}{h_{{1,2}}^{2}}h_{{2,1}}{h_{{2,2}}^{2}}+143616\,h_{{0,3}}h_{{1,2}}h_{{1,4}}{h_{{2,1}}}^{3}\\
&&-60480\,h_{{0,3}}{h_{{1,2}}^{2}}h_{{2,1}}h_{{4,2}}+177408\,h_{{0,4}}{h_{{1,2}}^{3}}h_{{2,1}}h_{{3,1}}+266112\,h_{{0,4}}{h_{{1,2}}^{2}}{h_{{2,1}}^{2}}h_{{2,2}}\\
&&+143616\,h_{{0,4}}h_{{1,2}}h_{{1,3}}{h_{{2,1}}}^{3}-48384\,h_{{0,4}}{h_{{1,2}}^{2}}h_{{2,1}}h_{{4,1}}+430848\,h_{{0,3}}h_{{0,4}}h_{{1,2}}{h_{{2,1}}^{2}}h_{{3,1}}\\
&&-39168\,h_{{0,4}}h_{{1,3}}{h_{{2,1}}}^{2}h_{{3,1}}+10752\,h_{{0,4}}h_{{1,2}}h_{{2,1}}h_{{5,1}}-103680\,h_{{0,3}}h_{{1,2}}h_{{1,3}}h_{{2,1}}h_{{4,1}}\\
&&+10752\,h_{{0,4}}h_{{2,1}}h_{{2,2}}h_{{4,1}}+10752\,h_{{0,4}}h_{{2,1}}h_{{3,1}}h_{{3,2}}-120960\,h_{{0,3}}h_{{1,2}}h_{{2,1}}h_{{2,2}}h_{{3,2}}\\
&&+221760\,{h_{{1,2}}^{3}}h_{{1,3}}h_{{2,1}}h_{{2,2}}+570240\,{h_{{0,3}}^{2}}h_{{1,2}}h_{{2,1}}h_{{2,2}}h_{{3,1}}-48384\,h_{{0,4}}h_{{1,2}}{h_{{2,1}}^{2}}h_{{3,2}}\\
&&+570240\,h_{{0,3}}{h_{{1,2}}^{2}}h_{{1,3}}h_{{2,1}}h_{{3,1}}+570240\,h_{{0,3}}h_{{1,2}}h_{{1,3}}{h_{{2,1}}^{2}}h_{{2,2}}+10752\,h_{{0,4}}h_{{1,2}}h_{{3,1}}h_{{4,1}}\\
&&-103680\,h_{{0,3}}h_{{1,3}}h_{{2,1}}h_{{2,2}}h_{{3,1}}-103680\,h_{{0,3}}h_{{1,2}}h_{{2,1}}h_{{2,3}}h_{{3,1}}-34560\,h_{{0,5}}h_{{1,2}}{h_{{2,1}}^{2}}h_{{3,1}}\\
&&\left.-96768\,h_{{0,4}}h_{{1,2}}h_{{2,1}}h_{{2,2}}h_{{3,1}} \right)\\
\bar a_{2,2}&=&8\,h_{{1,2}}h_{{1,3}}-24\,{h_{{0,3}}}^{2}h_{{2,1}}+8\,h_{{0,4}}h_{{2,1
}}+8\,h_{{0,3}} \left( -3\,{h_{{1,2}}}^{2}+h_{{2,2}} \right) -2\,h_{{2
,3}}\\
\bar a_{2,3}&=&1/2\,\sqrt {2} \left( -70\,h_{{0,4}}{h_{{1,2}}}^{2}+10\,{h_{{1,3}}}^{2
}+20\,h_{{1,2}}h_{{1,4}}+210\,{h_{{0,3}}}^{3}h_{{2,1}}+20\,h_{{0,5}}h_
{{2,1}} -4\,h_{{2,4}} \right.\\
&&\left.+35\,{h_{{0,3}}}^{2} \left( 9\,{h_{{1,2}}}^{2}-2\,h_{{2,2}}
 \right) +20\,h_{{0,4}}h_{{2,2}}-20\,h_{{0,3}} \left( 7\,h_{{0,4}}h_{{
2,1}}+7\,h_{{1,2}}h_{{1,3}}-h_{{2,3}} \right)\right)
\end{eqnarray*}



\begin{thebibliography}{99}

\bibitem{c17} A. Atabaigi, H. R. Z. Zangeneh, R. Kazemi, {Limit cycle bifurcation by perturbing a cuspidal loop of order 2 in a Hamiltonian system}, Nonlinear Anal. 75 (2012) 1945-1958.

\bibitem{c22} A. Atabaigi, H. R. Z. Zangeneh, {Bifurcation of limit cycles in small perturbations of a class of hyper-elliptic Hamiltonian systems of degree 5 with a cusp}, J. Appl. Anal. Comput.,  Vol. 1, No. 3 (2011): 299-313.

\bibitem{c1} M. Han, {Bifurcation Theory of Limit Cycles}, Sci. Press, 2013.

\bibitem{c3} M. Han, {Asymptotic expansions of Melnikov functions and limit cycle bifurcations},  Int. J. Bifurcat. Chaos,  Vol. 22, No.12 (2012) 1250296 (30 pages).

\bibitem{c7} M. Han, {On the number of limit cycles bifurcating from a homoclinic or heterclimic loop}. Sci. China Ser. A, 1993, {\bf36} (2):  113-132.

\bibitem{c8} M. Han, {Cyclicity of planar homoclinic loops and quadratic integrable systems}. Sci. China Ser. A, 1997, {\bf40} (12): 1247-1258.

\bibitem{c9} M. Han, J. Chen, {The number of limit cycles bifurcating from a pair of homoclinic loops}. Sci. China Ser. A, 2000, {\bf30} (5): 401-414.

\bibitem{c16} M. Han, J. Jiang, H. Zhu, {Limit cycles in a planar Hamiltonian systems with a nilpotent center}.  Int. J. Bifurcat. Chaos, 2008,  {\bf18} (10): 3013-3027.

\bibitem{c14} M. Han, C. Shu, J. Yang, A. Chian, {Polynimial Hamiltonian systems with a nilpotent critical point}. Advances in Space Research  2010, {\bf46}: 521-525.

\bibitem{c10} M. Han, J. Yang, A. Tarta, Y. Gao, {Limit cycles near homoclinic and heterclimic loop}. J. Dyn. Diff. Equat. (2008) {\bf20}: 923-944.

\bibitem{c2} M. Han, J. Yang, D. Xiao, {Limit cycle bifurcations near a double homoclinic loop with a nilpotent saddle}, Int. J. Bifurcat. Chaos,  Vol. 22, No.8 (2012) 1250189 (33 pages).

\bibitem{c11} M. Han, J. Yang, P. Yu, {Hopf bifurcations for near-Hamiltonian systems}. Int. J. Bifurcat. Chaos, Vol. 19, No. 12 (2009) 4117¨C4130.

\bibitem{c12} M. Han, P. Yu, {Normal forms, Melnikov functions and bifurcations of limit cycles}(Applied Mathematical Sciences Vol.181). Springer-Verlag, 2012.

\bibitem{c15} M. Han, H. Zang, J. Yang, {Limit cycle bifurcations by perturbing a cuspidal loop in a Hamiltonian syatem}. J. Differ. Equations
             {\bf246} (2009) 129-163.

\bibitem{c13} Y. Hou, M. Han, {Melnikov functions for planar near-Hamiltonian systems and Hoph bifurcations}. J. Shanghai Normal Universty(Natural Sciences), 2006, {\bf35} (1): 1-10.

\bibitem{c4} J. Li, {Hilbert's 16th problem and bifurcations of planar polynomial vector fields}, Int. J. Bifurcat. Chaos, 2003, {\bf13} (1): 47-106.

\bibitem{c18} J. Li, T. Zhang, M. Han, {Bifurcation of limit cycles from a heterclinic loop with two cusps}, Chaos, Solitons \& Fractals, Vol. 62-63,  2014, 44-54.

\bibitem{c6} R. Roussarie, {On the number of limit cycles which appear by perturbation of separatrix loop of planar vector fields}. Bol. Soc. Brasil. Mat., 1986, {\bf17} (2): 67-101.

\bibitem{c19} X. Sun, M. Han, J. Yang, {Bifurcation of limit cycles from a heterclinic loop with a cusp}, Nonlinear Anal. {\bf74} (2011) 2948-2965.

\bibitem{c21} J. Su, J. Yang, M. Han, {Hopf bifurcation of Li{\'e}nard systems by perturbing a nilpotent center},  Internat. J. Bifur. Chaos Appl. Sci. Engrg.,  {\bf22}, 1250203 (2012) [7 pages].

\bibitem{c20} H. Tian, M. Han, {Limit cycle bifurcations by perturbing a compound loop with a cusp and a nilpotent saddle}, Abstr. Appl. Anal. Vol. 2014, 15 pages.

\bibitem{c5} J. Yang, M. Han, {Limit cycles near a double Homoclinic loop}. Ann. Diff. Eqs. {\bf23}: 4 (2007), 536-545.

\bibitem{c23} J. Yang, {On the limit cycles of a kind of Li\'{e}nard system with a nilpotent center under perturbations}, J. Appl. Anal. Comput.  Vol. 2, No. 3 (2012): 325-339.


\end{thebibliography}
\end{document}